\documentclass[11pt,thmsa]{article}

 \DeclareMathAlphabet{\mathpzc}{OT1}{pzc}{m}{it}

{

\setlength\tabcolsep{0.4cm}

\usepackage{amssymb,latexsym}
\usepackage{amsmath,amscd}
\usepackage[all]{xy}
 \usepackage{amsmath}
 \usepackage{color}
 \usepackage{graphicx, pb-diagram}
\usepackage{amsthm} \usepackage{amsfonts}
 \usepackage{amssymb}
 \usepackage{wrapfig}
\usepackage{layout}
 \usepackage{verbatim}
 \usepackage{alltt}
 \usepackage{xypic}
 \input xy
 \xyoption{all}

\newtheorem*{theorem1}{Theorem \ref{theorem: automorphism}}

\newtheorem*{theorem3}{Theorem \ref{functor}}
 \newtheorem{theorem}{Theorem}[section]
 \newtheorem{lemma}[theorem]{Lemma}
 \newtheorem{proposition}[theorem]{Proposition}
 
 \newtheorem{corollary}[theorem]{Corollary}
 \newtheorem{definition}[theorem]{Definition}
  \theoremstyle{definition}
 
 \newtheorem{remark}[theorem]{Remark}

\newtheorem*{acknowledgements}{Acknowledgements}

\renewenvironment{proof}{\noindent{\it
Proof.}}{\bgroup\hspace{\stretch{1}}$\square$\egroup\medskip\par}

\newcommand{\Rep}{\mathrm{Rep}}

\newcommand{\hol}{\mathrm{hol}}

\newcommand{\Aut}{\mathsf{Aut}}

\newcommand{\MC}{\mathsf{MC}}

\newcommand{\id}{\mathrm{id}}
\newcommand{\End}{\mathrm{End}}
\newcommand{\Hom}{\mathrm{Hom}}

\newcommand{\ev}{\mathrm{ev}}

\newcommand{\A}{\mathsf{A}}

\newcommand{\FlatZ}{\mathcal{F}_\mathbb{Z}}
\newcommand{\basedloop}{\Omega_*}
\renewcommand{\loop}{\mathcal{L}}

\newcommand{\HC}{\mathsf{HC}}
\newcommand{\HH}{\mathsf{HH}}
\newcommand{\s}{\mathsf{s}}
\newcommand{\Deg}{\mathsf{D}}
\newcommand{\N}{\mathsf{N}}
\newcommand{\Moore}{\mathsf{M}}
\newcommand{\EZ}{\mathsf{EZ}}

\newcommand{\chains}{C_\bullet}

\newcommand{\Ainfty}{\mathsf{A}_\infty}
\newcommand{\C}{\mathcal{C}}
\newcommand{\D}{\mathcal{D}}
\newcommand{\dgFun}{\mathsf{dg}\mathrm{-}\mathsf{Fun}}
\newcommand{\Ob}{\mathsf{Ob}}

\newcommand{\G}{\mathsf{G}}
\newcommand{\barC}{\mathbb{B}}
\newcommand{\barH}{\mathbb{H}}

\begin{document}

\vspace{15cm}
 \title{Flat $\mathbb{Z}$-graded connections and loop spaces}

\author{Camilo Arias Abad\footnote{Escuela de Matemáticas, Universidad Nacional de Colombia, Medell\'in.  Partially supported
by the Swiss National Science Foundation at the University of Toronto. email: camiloariasabad@gmail.com.} \, and Florian Sch\"atz\footnote{Mathematics Research Unit, University of Luxembourg. Partially supported by the center of excellence grant ``Centre for Quantum Geometry of Moduli Spaces'' from the Danish National Research Foundation (DNRF95). email: florian.schaetz@gmail.com.}}

 \maketitle

 \begin{abstract} 
The pull back of a flat bundle $E\rightarrow X$ along the evaluation map $\pi: \loop X \rightarrow X$ from the free loop space $\loop X$ to $X$ comes equipped with a canonical automorphism given by the holonomies of $E$. This construction naturally generalizes to flat $\mathbb{Z}$-graded connections on $X$. Our main result is that the restriction of this holonomy automorphism to the based loop space $\basedloop X$ of $X$ provides an $\A_\infty$ quasi-equivalence
between the dg category of flat $\mathbb{Z}$-graded connections on $X$
and the dg category of representations of $\chains(\basedloop X)$,
the dg algebra of singular chains on $\basedloop X$.
\end{abstract}

\vspace{1cm}

\begin{flushright}
{\em Dedicated to James Dillon Stasheff, with gratitude.}
\end{flushright}

\vspace{0.5cm}

\tableofcontents

\section*{Introduction}\label{s:intro}
\addcontentsline{toc}{section}{\protect\numberline{}Introduction}

\subsection*{In a few words}

Suppose that $* \hookrightarrow X$ is a pointed manifold, $E$ is a flat vector bundle over $X$ and
$$\pi: \loop X=\{\gamma: S^1 \to X\} \rightarrow X$$
is the map given by evaluation at $ 1\in S^1\subset \mathbb{C}$. There is an automorphism of the pull back bundle $\pi^*E$, 
given by the holonomies of $E$.
We denote by $C_\bullet(\basedloop^\Moore X)$
the Pontryagin algebra of chains on the space of Moore loops in $X$. There is a natural map of dg algebras
\[ C_\bullet(\basedloop^\Moore X) \rightarrow H_0(\basedloop^\Moore X) \cong H_0(\basedloop X)\cong \mathbb{R}[\pi_1(X,*)].\]
The restriction of the holonomy automorphism to the based loop space gives $E\vert_*$ the structure of a representation of $\pi_1(X,*)$ and therefore the structure of a module over the group ring $\mathbb{R}[\pi_1(X,*)]$.
Via the morphism above, the vector space $E\vert_*$ becomes a module over the Pontryagin algebra. Thus, flat connections produce, via holonomies, the simplest kind of dg modules over the Pontryagin algebra.

We show that one can obtain
essentially all dg modules over the Pontryagin algebra 
if one passes to flat $\mathbb{Z}$-graded connections over $X$ and considers their generalized holonomies.
Even better, this construction yields a weak equivalence ($\A_\infty$ quasi-equivalence) between the dg category $\FlatZ(X)$ of flat $\mathbb{Z}$-graded
vector bundles over $X$ and the dg category of representations
of $\chains(\basedloop^\Moore X)$, see Theorem \ref{functor} in
\S \ref{subsection: automorphisms - functor version}.

\subsection*{In more words}
There are many equivalent points of view on the notion of a local system over a manifold $X$.
Some of the possibilities are listed in the following table:

\begin{center}
  \begin{tabular}{| c | c | c |}
\hline
   & {\bf point of view}   & {\bf local system}   \\ \hline
   {\bf I}& infinitesimal & flat connection \\ 
 {\bf T} &  topological {\footnotesize{(\& simplicial)}}& representation of $\pi_1(X) $   \\ 
{\bf H}  &homological & locally constant sheaf \\ 
 {\bf L} & free loop space &  automorphism over $\loop X$  \\
 \hline
  \end{tabular}
\end{center}

A higher version of the notion of a local system, called $\infty$-local system, has appeared in several recent works, including Block-Smith \cite{BS}, Arias Abad-Sch\"atz \cite{AS}, Holstein \cite{Holstein} and Malm \cite{Malm_thesis}. The word {\it higher} refers to the fact that, while local systems correspond to representations of the fundamental groupoid $\pi_1(X)$, $\infty$-local systems correspond to representations of the whole homotopy type of $X$ i.e. the infinity groupoid $\pi_\infty(X)$ of $X$. It turns out that all the different points of view on local systems mentioned above generalize to the $\infty$-local system case:

\begin{center}
  \begin{tabular}{| c |  c | c | }
\hline 
   & {\bf point of view}   & {\bf $\infty$-local system}   \\ \hline
   {\bf I}& infinitesimal & flat $\mathbb{Z}$-graded connection \\ 
{\bf S} & simplicial & representation of $\pi_\infty(X)$\\ 
{\bf T} &  topological & representation of $C_\bullet(\basedloop^\Moore X) $   \\ 
{\bf H}  &homological & homotopy locally constant sheaf \\ 

 {\bf L} & free loop space &  $\A_\infty$-automorphism over $\loop X$ \\ 
 \hline
  \end{tabular}
\end{center}

The fact that the different points of view on local systems are equivalent can be given a precise meaning:
each of the points of view can be used to define a category of local systems, and all the resulting categories are equivalent. Interestingly, there is a similar situation for $\infty$-local systems: each of the points of view
{\bf I}, {\bf S}, {\bf T}, {\bf H} defines a {\em dg} category and all these {\em dg} categories are equivalent. The equivalence of {\bf I} and {\bf S} was proved by Block-Smith \cite{BS} extending ideas of Igusa \cite{Igusa}. The equivalence between $\bf{S},{\bf T}$ and ${\bf H}$ was proved by Holstein \cite{Holstein}, using the homotopy theory of dg categories.
The present work is concerned with point of view ${\bf L}$ and a direct comparison between points of view  ${\bf I}$ and ${\bf T}$.

Let us be more explicit about what is meant by point of view ${\bf L}$. We will denote by $\pi: \loop X \rightarrow X$ the map given by evaluation at $1 \in S^1\subset \mathbb{C}$. Given a vector bundle $ E$ over $X$, the pull back bundle $\pi^*E$ is naturally a flat vector bundle over $\loop X$. Moreover, there is a natural automorphism $\hol \in \Aut (\pi^*E)$ given by the holonomies of $E$. Also, the automorphism $\hol$, when seen as a section of $ \End (\pi^*E)$ is covariantly constant. Hence, $\hol$ is an automorphism of the bundle $\pi^*E$ together with its flat connection. The following result asserts that the same holds in the case of $\infty$-local systems:

\begin{theorem1}
Let $E\rightarrow X$ be a flat graded vector bundle and denote by $\pi: \loop X\rightarrow X$ the natural projection given by evaluation at $1\in S^1$. Then $\pi^*E$ is a flat graded vector bundle over $\loop X$ which has a distinguished automorphism $\hol(E) \in \Aut(\pi^*E)$, characterized by the property that
if $U \subseteq X$  is an open subset over which $E$ trivializes, i.e.
$E\vert_U \cong U\times V$, we have
\[ \hol(E)|_{\loop U} = \sigma_{\End V}(1\otimes  e^{\alpha}),\]
where $d_\nabla=d +\alpha$.
Moreover, if $f: Y \rightarrow X$ is a smooth map then
\[ \hol(f^*E) =(\loop f)^*\hol(E),\]
with $\loop f: \loop Y \to \loop X$ given by $\loop f(\gamma):=f\circ \gamma$.
\end{theorem1}

Here, $\sigma_{\End V}$ is a twisted version of Chen's iterated integral map $\sigma$, see \cite{Chen} and \cite{GJP}. Variants of $\hol(E)$ appeared previously in the literature, see
 \cite{AlbertoCarlo, AZ} in the context of principal bundles,
 and \cite{Igusa} in the context of $\mathbb{Z}$-graded connections.
 The interpretation of $\hol(E)$ as a morphism in the dg category
 $\FlatZ(\loop X)$ seems to be new.

Our next result is a categorical enhancement of Theorem \ref{theorem: automorphism}.
Pulling back along the evaluation map $\pi: \loop X \to X$
induces a functor of dg categories
$\pi^*: \FlatZ(X) \to \FlatZ(\loop X)$.
It turns out that the generalized holonomies
of any flat $\mathbb{Z}$-graded vector bundle over $X$
can be interpreted as an $\A_\infty$-automorphism $\underline{\hol}$
of the dg functor $\pi^*$. For the precise statement
we refer to Theorem \ref{theorem: automorphism - categorical}
in \S \ref{subsection: automorphisms - functor version} below.

Next, we explain the relationship between points of view ${\bf I}$ and ${\bf T}$ on higher local systems. According to point of view ${\bf T}$, an $\infty$-local system is a finite cochain complex of vector spaces $(E,\partial)$ with the structure of a dg module over the Pontryagin algebra $C_\bullet(\basedloop^\Moore X)$, the dg algebra so singular chains on the based Moore loops. Let $E$ be a flat graded vector bundle over $X$ and denote by $\hol(E)$ the automorphism of $\pi^*E$ provided by Theorem \ref{theorem: automorphism}. We show that by restricting the automorphism $\hol(E)$ to the based loop space, the fiber of $E$ over the base point $*$ inherits the structure of a dg module over the Pontryagin algebra. This result provides an explicit construction of an infinity local system in the sense of ${\bf T}$, given a flat $\mathbb{Z}$-graded connection over $X$. Hence the passage from ${\bf I}$ to ${\bf L}$ and further to ${\bf T}$ produces an explicit $\A_\infty$-functor $\tilde{\varphi}$ of dg categories. Our main result
about this $\A_\infty$-functor is:

\begin{theorem3}
 Let $X$ be a connected manifold.
 The $\Ainfty$-functor
 $$ \tilde{\varphi}: \FlatZ(X) \to \Rep(\chains(\basedloop^\Moore X))$$
 is a quasi-equivalence, i.e.
 \begin{itemize}
 \item
 the chain maps given by $\hat{\varphi}_1$ induce isomorphisms between the cohomologies
 of the Hom-complexes,
 \item the induced functor between the homotopy categories is an equivalence.
 \end{itemize}
\end{theorem3}
 
 This result can be seen as an explicit implementation of the equivalence
 between the points of view ${\bf I}$ and ${\bf T}$ for $\infty$-local systems in terms of Chen's iterated integrals. The {\em existence} of such an equivalence can be also deduced by combining the equivalences
 ${\bf I} \sim {\bf S}$ (Block-Smith \cite{BS}) and
 ${\bf S} \sim {\bf H} \sim {\bf T}$ (Holstein \cite{Holstein}).
 Our proof that the $\Ainfty$-functor $\tilde{\varphi}$ is an quasi-equivalence relies on the standard Riemann-Hilbert correspondence for flat vector bundles and on
 a result due to Felix, Halperin and Thomas \cite{FHT} on
 the Hochschild complex of the Pontryagin algebra.
 \\
 
One of our motivations for writing this paper was to 
bring into closer contact Chen's iterated integrals and the point of view of derived geometry. Most notably,
Holstein's article \cite{Holstein} was a direct source of inspiration for the present work, as was Ben-Zvi and Nadler's \cite{BZN}.

\subsection*{Organization of the paper}

In \S \ref{section: Hochschild} we recall various Hochschild complexes, the dg category of representations $\Rep(A)$ of a dg algebra $A$
and the dg category $\FlatZ(X)$ of flat $\mathbb{Z}$-graded connections over a manifold $X$.
In \S \ref{section: formal holonomies} we define formal holonomies
on the level of the bar complex of a dg algebra and establish a formal version of gauge-invariance.
Section \S \ref{section: automorphisms} relates the formal holonomies to the free loop space $\pi: \loop X \to X$ of a manifold $X$ by means of Chen's iterated integrals.
We construct the generalized holonomy $\hol(E)$ of a flat $\mathbb{Z}$-graded vector bundle $E\to X$ and interpret it as an automorphism of the flat pull back bundle $\pi^*E \in \FlatZ(\loop X)$. We extend this point of view and obtain an $\A_\infty$-automorphism $\underline{\hol}$ of the dg
functor $\pi^*: \FlatZ(X)\to \FlatZ(\loop X)$.
In \S \ref{section: based loops} we first use $\hol(E)$
to obtain a representation of the Pontryagin algebra of $X$, i.e. the dg algebra of chains on the based (Moore) loop space
$\basedloop^\Moore X$.
We then recast the whole $\A_\infty$-automorphism $\underline{\hol}$ as an $\A_\infty$-functor
$\tilde{\varphi}: \FlatZ(X) \to \Rep(\chains(\basedloop^\Moore X))$
and finally establish that $\tilde{\varphi}$ is a quasi-equivalence.
The proof of this fact uses a technical -- and well-known --
result which is proved in Appendix \ref{appendix: HPT}.

\subsubsection*{Conventions}
\begin{itemize}
\item Let $W$ be a graded vector space. The suspension of $W$, denote by $\s W$, is the graded vector space defined by
\[ (\s W)^k:=W^{k+1}.\] 
We will use cohomological grading conventions i.e. all complexes have coboundary operators which increase the degree by $1$ and all dg algebras have differntials of degree $1$. In particular,
the complex of singular chains $\chains(X)$ for $X$ a topological space
is concentrated in non-positive degrees.
\item We assume that all dg algebras come with a unit. If $A$ is a dg algebra, the commutator $[a,b]$ of two elements $a,b \in A$ will always mean the graded commutator, that is:
\[ [a,b]=ab-(-1)^{|a||b|}ba.\]
\item Given a dg algebra  $A$ we will denote by $\MC(A)$ the set of Maurer-Cartan elements in $A$, i.e.
\[ \MC(A):=\{ a\in A^1: da + a\cdot a=0\}.\]

\item All vector bundles are assumed to be complex.

\item Given a dg category $\mathcal{C}$, we denote by $\mathsf{H}^0\mathcal{C}$ its homotopy category, i.e. the category that has the same objects as $\mathcal{C}$ and satisfies:
\[ \Hom_{\mathsf{H}^0\C}(X,Y)=H^0(\Hom(X,Y)).\]
\end{itemize}

\acknowledgements{We would like to thank David Mart\'inez Torres for his help with the proof of Lemma \ref{lemma: skeleton} and Dan Christensen, Peter Teichner and Konrad Waldorf for helpful discussions. Camilo Arias Abad would also like to thank Marco Gualtieri, Alberto Garc\'ia-Raboso and the other participants of the Glab seminar at the Fields Institute for interesting discussions on related subjects.}

\section{Preliminaries}\label{section: Hochschild}

In this section all dg algebras are assumed to have an augmentation $\epsilon:A \rightarrow \mathbb{R}$.
We denote by $\bar{A}$ the augmentation ideal $\ker \epsilon$.

\subsection{Hochschild complexes}\label{subsection: Hochschild}

\begin{definition}
Let $A$ be a dg algebra and $M$ a dg bimodule over $A$.
 The Hochschild chain complex $\HC_\bullet (A,M)$ is the graded vector space
 $$ \HC_\bullet(A,M):= \bigoplus_{k\ge 0}  M\otimes (\s \bar{A})^{\otimes k},$$
 equipped with the boundary operator $b=b_0 + b_1$, where 
 \[ b_0(m\otimes a_1\otimes  \cdots \otimes a_n)=dm\otimes a_1\otimes \cdots \otimes a_n+ \sum_{i=1}^n (-1)^{|m|+\cdots +|a_{i-1}|-i} m\otimes \cdots \otimes d a_i \otimes \cdots \otimes a_n ,\]
 and
\begin{eqnarray*}
b_1(m\otimes \cdots \otimes a_n)&=& (-1)^{|m|+1}ma_1 \otimes \cdots \otimes a_n+\\
&& + \sum_{i=1}^{n-1}(-1)^{|m|+\cdots +|a_{i}|-i+1} m\otimes \cdots \otimes (a_i a_{i+1}) \otimes \cdots \otimes a_n\\
&&+(-1)^{(|m|+\cdots +|a_{n-1}|-n-1)(|a_{n}|-1)}a_n m \otimes \cdots \otimes a_{n-1}.
 \end{eqnarray*}
 The cohomology of this complex, denoted $\HH_{\bullet}(A,M)$, is called the Hochschild homology of $A$ with coefficients in $M$.
\end{definition}

\begin{remark}
In the special case where $M=\mathbb{R}$ is the trivial bimodule we will write $\HC_\bullet(A)$ instead of $\HC_\bullet(A,\mathbb{R})$ and $\HH_{\bullet}(A)$ instead of $\HH_{\bullet}(A,\mathbb{R})$.
In the case where $M=A$ we obtain the bar complex of $A$, which we denote by $\barC A$ instead of $\HC(A,A)$. Similarly we write $\mathbb{H}_\bullet(A)$ instead of $\HH_\bullet(A,A)$.
\end{remark}

\begin{remark}
The Hochschild complex $\HC_\bullet(A)$ comes equipped with a comultiplication
$\Delta$
given by
$$ 
\Delta( a_1\otimes \cdots \otimes  a_k) = \sum_{i=0}^{k}( a_1 \otimes \cdots \otimes  a_i)
\otimes (a_{i+1} \otimes \cdots \otimes a_k).
$$ 
\end{remark}

Following \cite{GJP}, we note that if $A$ is graded commutative, the bar complex $\barC A$ has a natural product $*$ given by
\begin{eqnarray*}
(\beta\otimes  a_1 \otimes \cdots \otimes  a_k) &*&
(\gamma\otimes  a_{k+1} \otimes \cdots \otimes  a_{l+k})
=\\
&=&
 (-1) ^{ (|a_1|+\cdots +|a_k|-k)|\gamma|} \sum_{\chi \in S_{k,l}}[a_1 \otimes \cdots \otimes a_{k+l}](\chi)  \beta \gamma \otimes a_{\chi(1)}\otimes  \cdots \otimes a_{\chi_{k+l}},
 \end{eqnarray*}
where $S_{k+l}$ is the symmetric group, $S_{k,l}$ the set of shuffle permutations and
 $$[a_1 \otimes \cdots \otimes a_{k+l}]: S_{k+l} \rightarrow \{1,-1\}$$
 is the Koszul character sign, i.e. the unique character such that:
 $$[a_1 \otimes \cdots \otimes a_{k+l}](\mu)= (-1)^{(|a_i|-1)(|a_{i+1}|-1)},$$
 if $\mu$ is the transposition $(i,i+1)$.
The product $*$ is associative, graded commutative and compatible with $b$, and hence gives $\barC A$ the structure of a  commutative dg algebra.
Observe that the inclusion $i: A\hookrightarrow \barC A$ is a morphism of dg algebras.
The coboundary operator $b$, the product $*$ and the coproduct $\Delta$ then
make $\HC_\bullet(A)$ into a dg Hopf algebra.

\begin{definition}
Let $A$ be a dg algebra and $M$ a dg bimodule over $A$.
The Hochschild cochain complex $\HC^\bullet(A,M)$ of $A$ with values in $M$ is the cochain complex
$$ \HC^\bullet(A,M) = \Hom(\HC_\bullet(A),M),$$
with differential $b$ defined by the formula:
\begin{eqnarray*}
b(\eta)( a_1 \otimes\cdots \otimes a_k)&:= &d (\eta(a_1\otimes \cdots \otimes a_k))-(-1)^{|\eta|} \eta(b(a_1\otimes \cdots \otimes a_k))\\&&+(-1)^{|\eta|+|a_1|+\cdots +|a_{k-1}|+k}\eta(a_1\otimes
\cdots \otimes a_{k-1})a_k\\
&&+ (-1)^{(|\eta|+1)(|a_1|+1)+1}a_1\eta(a_2 \otimes \cdots \otimes a_k).
\end{eqnarray*}
The Hochschild cohomology $\HH^\bullet(A,M)$ of $A$ with values in $M$
is the cohomology of $\HC^\bullet(A,M)$.
\end{definition}

\begin{definition}
Let $B$ be a dg algebra, seen as a dg module over $A$ via the augmentation map.
The Hochschild cochain complex $\HC^\bullet(A,B)$ of $A$ with values in $B$ is equipped with a cup product
$$
(\varphi \cup \psi)( a_1 \otimes \cdots \otimes  a_{p+q}) :=
(-1)^{|\psi|(|a_1|+\cdots + |a_{p}|-p)}\varphi( a_1\otimes \cdots \otimes  a_p) \cdot \psi( a_{p+1} \otimes \cdots \otimes  a_{p+q}).
$$
The cup product is compatible with the differential $b$ and gives $\HC^\bullet(A,B)$ the structure of a dg algebra.
\end{definition}

\begin{remark}
 The cup product can be expressed in terms of the coproduct $\Delta$ in $\HC(A)$ and the multiplication map $m_B: B\otimes B\to B$ of $B$ as follows:
   $$ \varphi \cup \psi := m_B\circ (\varphi\otimes \psi) \circ \Delta.$$
\end{remark}

\subsection{Representations of dg algebras}\label{subsection: reps of dg algebras}

Let $A$ and $B$ be dg algebras. Denote by $\mathsf{TwHom}(A,B)$ the set of pairs $(\beta,\varphi)$,  where
$\beta \in \MC(B)$ and $\varphi: A \rightarrow B_\beta$ is a unital $\A_\infty$-morphism. Here $B_\beta$ is the algebra $B$ with twisted differential $d+[\beta,-]$.

\begin{lemma}
 There is a natural bijection between 
$ \MC(\HC^\bullet(A,B))$ and  $\mathsf{TwHom}(A,B)$.
\end{lemma}

\begin{proof}
 Let $\mu$ be a Maurer-Cartan element of $\HC^\bullet(A,B)$.
 We denote by $\mu_{(k)}$ the restriction of $\mu$ to $(\s \bar{A})^{\otimes k}$, i.e.
 $$ \mu = \sum_{k\ge 0} \mu_{(k)}, \quad \mu_{(k)}: (\s \bar{A})^{\otimes k} \to B.$$
 
 The element $\beta:=\mu_0(1) \in B^1$ satisfies the Maurer-Cartan equation. Then we set $\varphi:=\{ \varphi_k\}_{k\geq 1}$, where:
 \[ \varphi_k: (\s A)^{\otimes k}\rightarrow \s B,\]
 is defined as follows:
  \begin{itemize}
   \item For $k=1$, $\varphi_{(1)}(a)=\s(\mu_1(a))$ for $a \in \s\bar{A}$ and $\varphi_1(\s1)=\s 1$.
   \item For $k>1$, $\varphi_{(k)}:=\s\circ \mu_k \circ p$ where 
   $p:(\s A)^{\otimes k} \to (\s \bar{A})^{\otimes k}$, is the projection map that uses the identification $\bar{A} \cong A / \mathbb{R}$.
  \end{itemize}
 It is a standard computation that $\varphi: A \rightarrow B_\beta$ is an $\mathsf{A}_\infty$-morphism. Since $\mu$ can be recovered from the pair $(\beta,\varphi)$ we conclude that the assignment:
 \[ \mu \mapsto (\beta, \varphi),\]
 is bijective.
  
\end{proof}

\begin{definition}
We denote by $\MC_+(\HC^\bullet(A,B))$ those Maurer-Cartan elements of $(\HC^\bullet(A,B),b,\cup)$
whose restriction to $\mathbb{R} = (\s \bar{A})^{\otimes 0} \subset \HC_\bullet(A)$ is zero.
\end{definition}

\begin{corollary}
There is a one-to-one correspondence between elements of $\MC_+(\HC^\bullet(A,B))$
and $\Ainfty$-morphisms from $A$ to $B$.
\end{corollary}

\begin{definition}
Let $M$ be a dg $B$-$B'$-module and $\tau$, $\tau'$ elements of $\MC_+(\HC^\bullet(A,B))$
and $\MC_+(\HC^\bullet(A,B'))$, respectively. We can then twist the coboundary operator 
$b$ on $\HC^\bullet(A,M)$ by
$$ b_{\tau,\tau'}\varphi = b \varphi + \tau \cup \varphi  - (-1)^{|\varphi|} \varphi \cup \tau'.$$ 

\end{definition}

\begin{definition}
Let $A$ be a dg algebra. We denote by $\Rep(A)$
the dg category defined as follows.
The objects of $\Rep(A)$ are complexes $(V,\partial)$ of finite total dimension,
together with an element 
$$\tau \in \MC_+(\HC^\bullet(A,\End V)).$$
If $(V,\tau)$ and $(V',\tau')$ are objects of $\Rep(A)$ then:
$$\Hom_{\Rep(A)}((V,\tau),(V',\tau')):=\HC^\bullet(A,\Hom(V,V')),$$
with coboundary operator twisted by $\tau$ and $\tau'$, and composition given by the cup product.
\end{definition}

\begin{remark}

 In other words, objects of the dg category $\Rep(A)$ are $\Ainfty$-morphisms with domain $A$ and codomain a dg algebra
 of the form $(\End V, [\partial,-], \circ)$, with $(V,\partial)$ a complex of finite total dimension. Such morphisms are also known
 as $\Ainfty$-modules over $A$. Furthermore, the cocycles of the morphism complex $\HC^\bullet(A,\Hom(V,W))$
 in degree $0$ correspond exactly to morphisms of the $\Ainfty$-modules $V$ and $W$.
 The condition of such a cocycle to be a coboundary corresponds exactly to the notion of being null-homotopic in the sense of
 \cite{Keller}. 
 It is a standard fact that every object of $\Rep(A)$ is quasi-isomorphic to one of the form $((V,0),\tau)$, i.e. one where the differential on the complex is trivial, see \cite{Keller}. One way to prove this is via homotopy transfer, which allows one to replace a complex $(V,\partial)$ by its cohomology $H(V)$, seen as a complex with zero differential.
\end{remark}

\subsection{Flat $\mathbb{Z}$-graded connections}\label{section: Z-graded connections}

\begin{definition}
A $\mathbb{Z}$-graded vector bundle $E\to X$ is a complex vector bundle $E$ of the form
$$ E = \bigoplus_{k\in \mathbb{Z}}E_k,$$
where each of the $E_k$ is a complex vector bundle of finite rank and $E_k=0$ for $|k|>>0$. 
The space of differential forms with values in $E$ is:
\[ \Omega(X,E) := \bigoplus_{i,k}\Omega^i(X, E_k)=\bigoplus_{i,k}\Gamma(\Lambda^i(T^*X)\otimes E_k).\]
 We say that an element $\alpha \in \Omega^i(X,E_k)$ has form-degree $i$, internal degree $k$ and total degree $k+i$.
 We will be mostly interested in the total degree and write:
 \[ \Omega(X,E)^p:=\bigoplus_{i+k=p}\Omega^i(X,E_k).\]
 The vector space $\Omega(X,E)$ has the structure of a graded module over the algebra $\Omega(X)$.
\end{definition}

\begin{definition}
A $\mathbb{Z}$-graded connection $D$ on a graded vector bundle $E$ is a linear operator
$$ D: \Omega(X,E)^\bullet \to \Omega(X,E)^{\bullet +1}$$
which increases the total degree by $1$ and satisfies the Leibniz-rule
$$ D(\alpha \wedge \eta) = d\alpha \wedge \eta + (-1)^{|\alpha|}\alpha \wedge D\eta,$$
for any homogeneous  $\alpha \in \Omega^\bullet(X)$ and $\eta \in \Omega(X,E)$.
We say that  $D$ is flat if $D^2=0$.
\end{definition}

\begin{remark}
If $E$ is a trivial vector bundle  $ X\times V$, where $V$ is a graded vector space, there is a bijective correspondence 
\[\MC(\Omega(X)\otimes \End V)\leftrightarrow \{\text{flat $\mathbb{Z}$-graded connections on $E$}\}  \]
given by\[\alpha \mapsto d+[\alpha,-].\]
\end{remark}

\begin{definition}
If $X$ is a smooth manifold the dg category $\FlatZ(X)$ of flat $\mathbb{Z}$-graded connections is defined as follows. The objects of $\FlatZ(X)$ are $\mathbb{Z}$-graded vector
bundles endowed with flat $\mathbb{Z}$-graded connections. If $E,E'$ are objects of $\FlatZ(X)$ then:
\[ \Hom(E,F):=\Hom_{\Omega(X)-\mathsf{mod}}(\Omega(X,E),\Omega(X,E')) = \Omega(X,\Hom(E,E')),\]
which is a cochain complex with coboundary operator given by:
\[d\phi:=D' \circ \phi -(-1)^{|\phi|} \phi \circ D.\]
\end{definition}

For further details on the structure of these dg categories the reader may consult \cite{AS}.

\section{The bar complex and formal holonomies}\label{section: formal holonomies}

\subsection{Formal holonomies}\label{subsection: formal holonomies - bar}

The following well-known lemma -- see \cite{Algebraic-string-bracket} for instance -- explains how to associate a cochain in the completed bar complex $\hat{\barC} A$ of a dg algebra $A$
to a Maurer-Cartan element of $A$.

\begin{lemma}\label{lemma: MC}
Let $A$ be a dg algebra and $a \in \MC(A)$ a solution to the Maurer-Cartan equation.
Set \[1\otimes e^{ a}:=\sum_{k\geq 0} 1 \otimes  a^{\otimes k}\in  \hat{\barC} A,\]
where $\hat{\barC} A$ is the completion of $\barC A$, i.e.
$ \hat{\barC} A := \prod_{k \ge 0} A\otimes (\s \bar{A})^{\otimes k}$.
Then:
\[ b(1\otimes e^{ a})=0.\]
\end{lemma}
\begin{proof}
Let us compute:
\begin{eqnarray*}
b (\sum_{k\geq 0} 1\otimes a^{\otimes k})&=&  -\sum_{k \geq 0} \left(
\sum_{i=1}^k  1\otimes \cdots \otimes da \otimes \cdots \otimes a 
+\sum_{i=1}^{k-1}1 \otimes \cdots \otimes a \cdot \alpha \otimes \cdots \otimes a \right)\\
&=&- \sum_{k \geq 0} \left(\sum_{i=1}^k  1\otimes \cdots \otimes (da+\alpha\cdot a) \otimes \cdots \otimes a \right)=0.
\end{eqnarray*}

\end{proof}

\begin{proposition}\label{proposition: formal holonomies}
 Let $A$ be a commutative dg algebra and $K$ a graded algebra.
The linear map
 \begin{eqnarray*}
 \phi_K: \hat{\barC}(A\otimes K) & \to & \hat{\barC} A\otimes K,\\
 a_0\otimes k_0 \otimes \cdots \otimes  a_r\otimes k_r &\mapsto & (-1)^{\sum_{i<j} |k_i|(|a_j|+1)}a_0\otimes  a_1 \otimes \cdots \otimes  a_r\otimes k_0\cdots k_r
 \end{eqnarray*}
has the following properties:
\begin{enumerate}
\item If $i: A\otimes K \hookrightarrow \hat{\barC}(A\otimes K)$ and $j: A\otimes K \hookrightarrow \hat{\barC} A\otimes K$
 are the natural inclusions, then:
\[ \phi_K \circ \iota =j.\]
 \item For any $\beta \in A\otimes K$ of degree $1$ the identity
 $$ \phi_Kb(1\otimes e^{ \beta}) = b\phi_K(1\otimes e^{\beta}) + [j(\beta),\phi_{K}(1\otimes e^{\beta})]$$
 holds.
\end{enumerate}

\end{proposition}

\begin{proof}
 The first assertion is a simple check, so we pass to item (2).
 To obtain an expression for $\phi_K \circ b(1\otimes e^{\beta})$,
 we observe that $\phi_K$ factors as follows:
 $$
  \xymatrix{ \hat{\barC}(A\otimes K) \ar[r]^{\kappa} & \hat{\barC} A\otimes \mathsf{T}K \ar[r]^{\mathrm{id}\otimes m} & \hat{\barC} A\otimes K,}
  $$
  where $\mathsf{T}K$ is the tensor algebra on $K$, the first map $\kappa$ is the canonical isomorphism of graded vector spaces and the second map acts by multiplication
  $m(k_0\otimes \cdots \otimes k_r):=k_0\cdots k_r$.
 
  Our aim is to compute
  $$   \kappa b(1\otimes e^{ \beta}) - b\kappa(1\otimes e^{\beta}).$$
  Since $\kappa$ commutes with the operator $b_0$, we can replace $b$ by $b_1$ above.
  
  Writing $\beta = \sum_{j=1}^ma_j\otimes k_j$, we obtain that
  $\kappa b_1(1\otimes e^{ \beta})$ is the sum over $i_1,\dots, i_r \in \{ 1,\dots, m\}$ of the expression:
  \begin{eqnarray*}
   &&(-1)^{\tau+1}(a_{i_1}\otimes  a_{i_2}\otimes \cdots \otimes  a_{i_r}) \otimes (k_{i_1}\otimes \cdots \otimes k_{i_r}) +\\
    && \sum_{l=1}^{r-1}(-1)^{\tau+|k_{i_1}|+\cdots+|k_{i_l}|+1} (1\otimes  a_{i_1}\otimes \cdots \otimes a_{i_l}a_{i_{l+1}} \otimes \cdots \otimes  a_{i_r})\otimes (1\otimes k_{i_1}\otimes \cdots \otimes k_{i_l}k_{i_{l+1}}\otimes \cdots \otimes k_{i_r}) +\\
   && (-1)^{\tau}(a_{i_r}\otimes a_{i_1}\otimes \cdots \otimes  a_{i_{r-1}})\otimes (k_{i_r}\otimes k_{i_1}\otimes \cdots \otimes k_{i_{r-1}}),
  \end{eqnarray*}
  
 where $\tau:= \sum_{a<b} |k_{i_a}||k_{i_b}|$.  
 
 On the other hand, $b_1 \kappa(1\otimes e^{\beta})$ is the sum over $i_1,\dots, i_r \in \{ 1,\dots, m\}$ of the terms:
   \begin{eqnarray*}
   &&(-1)^{\tau+1 }(a_{i_1}\otimes  a_{i_2}\otimes \cdots \otimes  a_{i_r}) \otimes (1\otimes k_{i_1}\otimes \cdots \otimes k_{i_r}) +\\
    &&\sum_{l=1}^{r-1} (-1)^{\tau+|k_{i_1}|+ \cdots +|k_{i_l}|+1}(1\otimes  a_{i_1}\otimes \cdots \otimes a_{i_l}a_{i_{l+1}} \otimes \cdots \otimes  a_{i_r})\otimes (1\otimes k_{i_1}\otimes \cdots \otimes k_{i_l}\otimes k_{i_{l+1}}\otimes \cdots \otimes k_{i_r}) +\\
   &&(-1)^{\tau+|k_{i_r}|(|k_{i_1}|+\cdots + |k_{i_{r-1}}|)} (a_{i_r}\otimes a_{i_1}\otimes \cdots \otimes  a_{i_{r-1}})\otimes ( k_{i_1}\otimes  \cdots  \otimes k_{i_r}) .
  \end{eqnarray*}
  
 We conclude that the difference \[\phi_K b(1\otimes e^{ \beta}) -b \phi_K(1\otimes e^{\beta})\]    is the sum over all indices of the following terms:
  \begin{eqnarray*}
  (-1)^{\tau}(a_{i_r}\otimes a_{i_1}\otimes \cdots \otimes  a_{i_{r-1}})\otimes (k_{i_r} k_{i_1} \cdots k_{i_{r-1}})-\\
  (-1)^{\tau+|k_{i_r}|(|k_{i_1}|+\cdots + |k_{i_{r-1}}|)} (a_{i_r}\otimes a_{i_1}\otimes \cdots \otimes  a_{i_{r-1}})\otimes ( k_{i_1} \cdots  k_{i_r})   . \end{eqnarray*}
  This sum is precisely $[j(\beta), \phi_K(1\otimes e^{ \beta})]$.\end{proof}

\begin{corollary}\label{corollary: formal holonomies}
Let $A$ and $B$ be commutative dg algebras
and
$$ \varphi: \hat{\barC} A\to B$$
a morphism of dg algebras. 
Given a graded algebra $K$, we consider the linear map
 \[ \varphi_K: \hat{\barC}(A\otimes K) \rightarrow B\hat{\otimes} K\]
 defined as the composition:
 \[ \varphi_K:=    (\varphi \otimes \id) \circ \phi_K.\]
 For any element $\beta \in A\otimes K$ of degree $1$, the identity
 $$ \varphi_K b(1\otimes e^{ \beta}) = d\varphi_K(1\otimes e^{\beta}) + [(\varphi\otimes \mathrm{id})(\beta),\varphi_{K}(1\otimes e^{\beta})]$$
 holds.
 In particular, if $\beta$ is a Maurer-Cartan element of $A\otimes K$, then
 $$  d\varphi_K(1\otimes e^{\beta}) + [(\varphi\otimes \mathrm{id})(j(\beta)),\varphi_{K}(1\otimes e^{\beta})] = 0.$$
\end{corollary}

\begin{proof}
This is a formal consequence of Propositon \ref{proposition: formal holonomies} because $ (\varphi \otimes \mathrm{id})$ is a morphism of dg algebras.
\end{proof}

\subsection{Formal gauge-invariance}\label{subsection: formal holonomies - normalized bar}

\begin{remark} \label{gauge action}
Let $A$ be a dg algebra and denote by $U(A)$ the group of invertible elements of degree zero.
The group $U(A)$ acts on the right on $A$ by the formula
\[ x \bullet g:= g^{-1}x g+g^{-1}dg.\]
Moreover, this action restricts to the Maurer-Cartan set $\MC(A)$.

We want to understand the behaviour of the gauge-action under the assignment
\[
 (A\otimes K)^1  \to  \hat{\barC} A\otimes K, \qquad \beta \mapsto \phi_E(1\otimes e^{ \beta}).
\]
It turns out that one needs to pass to a certain quotient complex of $\barC A$ introduced by Chen
in his study of the algebraic topology of the loop space, see \cite{Chen}.
Chen studied this quotient in more detail in \cite{Chen_reduced-bar},
we follow the exposition in \cite{GJP}.
\end{remark}

\begin{definition}\label{definition: operators R and S}
Let $A$ be a dg algebra concentrated in non-negative degrees.
Given $f\in A^0$, and $i\ge 1$, we define operators $S_i(f)$ and $R_i(f)$ on $\barC A$ as follows:
$$S_i(f)(a_0\otimes a_1\otimes \cdots \otimes  a_k) := a_0\otimes  a_1 \otimes \cdots \otimes  a_{i-1} \otimes  f \otimes  a_{i} \otimes \cdots \otimes  a_k$$
and define $R_i(f)$ to be the commutator  $[b ,S_i(f)]$.
\end{definition}

\begin{definition}\label{definition: reduced bar}
Let $A$ be a dg algebra concentrated in non-negative degrees.
The subcomplex of degenerate chains of $\barC A$, denoted by $\Deg(A)$, is the linear span of the images
of $S_i(f)$ and $R_i(f)$, for all $f\in A^0$ and all $i\ge 1$.
The normalized Hochschild complex $\N(A)$ of $A$ is the quotient complex
$$ \N(A) = \barC A / \Deg(A).$$
\end{definition}

In the case where $A=\Omega^\bullet(X)$ is the algebra of differential forms on a manifold, passing to the quotient complex $\N(\Omega^\bullet(X))$
is natural, since $\Deg(\Omega^\bullet(M))$ coincides with the kernel of Chen's iterated integral map $\sigma$ which maps $\barC \Omega^\bullet (X)$ to $\Omega^\bullet(\loop X)$.
In \cite{Chen}, Chen shows that if $A$ is a dg algebra 
concentrated in non-negative degrees such that $H^0(A)=\mathbb{R}$, $\Deg(\Omega^\bullet(X))$ is acyclic.
 If $A$ is assumed to be graded commutative, $\Deg(A)$ is an ideal
with respect to the shuffle product, hence $\N(A)$ has the structure of a
commutative dg algebra. 

We will denote by $\hat{\N}(A)$ the space
\[ \hat{\N}(A):= \hat{\barC}(A)/\hat{\Deg}(A),\]
where $\hat{\Deg}(A)$ is the completion of $\Deg(A)$.
In case $A$ is graded commutative, the shuffle product extends in a natural way to give $\hat{\N}(A)$
the structure of a commutative dg algebra.

\begin{lemma}\label{conjugation}
 Let $A$ be a commutative dg algebra 
 and $K$ a graded algebra. For any $g \in A^0 \otimes K^0$ and any
  \[X=1 \otimes (a_1 \otimes k_1)\otimes \cdots \otimes (a_n 
 \otimes k_n) \in \mathbb{B}(A \otimes K)\] the following equalities hold in $\N(A) \otimes K$:
 \begin{eqnarray*}
 \phi_K( 1\otimes g(a_1\otimes k_1) \otimes \cdots \otimes (a_n \otimes k_n))-\phi_K(1 \otimes dg\otimes (a_1\otimes k_1) \otimes \cdots \otimes (a_n \otimes k_n)) &=& g \phi_K(X)\\
  \phi_K( 1\otimes (a_1\otimes k_1) \otimes \cdots \otimes (a_n \otimes k_n)g)+\phi_K(1 \otimes (a_1\otimes k_1) \otimes \cdots \otimes (a_n \otimes k_n)\otimes dg) &=&  \phi_K(X)g.  
  \end{eqnarray*}
  \end{lemma}
\begin{proof}
We will prove only the second equation, the first one is analogous. Since both sides are linear in $g$ we may assume that $g= a\otimes k$. Let us set 
\[ \epsilon=\sum_{i<j} |k_i|(|a_j|+1)\]
and compute:
 \begin{eqnarray*}
  \phi_K( 1\otimes (a_1\otimes k_1) \otimes \cdots \otimes (a_n \otimes k_n)g)+\phi_K(1 \otimes (a_1\otimes k_1) \otimes \cdots \otimes (a_n \otimes k_n)\otimes dg)& =&\\ 
\phi_K(   1\otimes (a_1\otimes k_1) \otimes \cdots \otimes (a_na \otimes k_n k))+\phi_K(1 \otimes (a_1\otimes k_1) \otimes \cdots \otimes (a_n \otimes k_n)\otimes (da\otimes k)) &=& \\ 
(-1)^{\epsilon}(   1\otimes (a_1\otimes\cdots \otimes a_na) \otimes (k_1\cdots k_n k)+(-1)^{\epsilon}(1 \otimes (a_1\otimes \cdots \otimes a_n \otimes da) \otimes (k_1 \cdots k_n k) &=&\\
(-1)^\epsilon a \otimes a_1 \otimes \cdots \otimes a_n \otimes ( k_1 \cdots k_n k)= \phi_K(X) g. \end{eqnarray*}
 Note that when passing to the last line, we have used that the computation takes place in  $\N(A) \otimes K$,  and that the equation
\[ 1\otimes a_1 \otimes \cdots \otimes (a_n a) + 1 \otimes a_1 \otimes \cdots \otimes a_n \otimes da =a \otimes a_1 \otimes \cdots \otimes a_n\]
holds in $\N(A)$.
 \end{proof}

\begin{proposition}\label{proposition: gauge-invariance}
 Let $A$ be a commutative dg algebra and $K$ a graded algebra.
We consider the map
\[
 \hat{\phi}_K: (A\otimes K)^1  \to  \N(A)\hat{\otimes} K, \qquad \beta \mapsto \phi_K(1\otimes e^{\beta}),
\]
 with $\phi_K$ defined as in Proposition \ref{proposition: formal holonomies}.
 For any invertible element $ g \in A^0 \otimes K^0$ we have:
 \[ \phi_K( 1 \otimes e^{\beta \bullet g})= g^{-1} \phi_K( 1 \otimes e^\beta)  g.\]
 The expression on the right is defined via the inclusion $A \otimes K \hookrightarrow N(A) \otimes K$.

\end{proposition}

\begin{proof}
We need to prove the equality
\[ \phi_K( 1 \otimes e^{\beta \bullet g})= g^{-1} \phi_K( 1 \otimes e^\beta)  g.\]
Expanding the exponential $ 1 \otimes e^{\beta \bullet g}=1 \otimes e^{g^{-1} \beta g+ g^{-1}dg}$ we obtain the expression:
\[1 \otimes e^{\beta \bullet g}=\sum_{r\geq0} 1 \otimes e^{g^{-1}dg}\otimes (g^{-1}\beta g \otimes e^{g^{-1}dg})^{\otimes r}.\]
Therefore, it is enough to prove that for $r$ fixed we have:
\begin{eqnarray}\label{fixed r}
{\phi}_K(1 \otimes e^{g^{-1}dg}\otimes (g^{-1}\beta g \otimes e^{g^{-1}dg})^{\otimes r})= g^{-1} {\phi}_K(1\otimes  \beta ^{\otimes r})  g.
\end{eqnarray}
To this end,
we set \[\N'(A\otimes K):=\mathbb{B}(A\otimes K)/ \Deg'(A\otimes K),\]
where $\Deg'(A \otimes K) \subseteq \Deg(A \otimes K)$  is the vector subspace spanned by the images of elements of the form
$S_i(f), R_i(f)$ where $f \in A^0 \otimes K^0$. The map ${\phi}_K: \mathbb{B}(A\otimes K) \rightarrow \N(A) \otimes K$ factors through the quotient $\N'(A \otimes K)$. We claim that for each $r\geq 2$ the following equality holds in 
$\N'(A \otimes K)$:
\begin{equation}\label{N'}
1 \otimes e^{g^{-1}dg}\otimes (g^{-1}\beta g \otimes e^{g^{-1}dg})^{\otimes r}= 1 \otimes e^{g^{-1}dg}\otimes g^{-1}\beta\otimes\beta^{\otimes r-2}\otimes \beta g \otimes  e^{g^{-1}dg}
\end{equation}
Since the computation occurs in $\N'(A\otimes K)$, we have the following identity that will be used repeatedly:
\begin{equation}\label{basic}
 \cdots x_{i-1} \otimes dg \otimes x_i \otimes \cdots \quad = \quad  \cdots\otimes x_{i-1} \otimes gx_i \otimes \cdots - \cdots \otimes x_{i-1} g \otimes x_i \otimes \cdots 
 \end{equation}
Let us now use Equation (\ref{basic}) to show that the following identity also holds in $N'(A\otimes K)$:
\begin{equation}\label{middle}
 \cdots \otimes x g \otimes e^{ g^{-1}dg} \otimes  g^{-1} y  \otimes \cdots \quad = \quad \cdots x \otimes y \otimes \cdots 
 \end{equation}
 
 For this we expand $\cdots \otimes  x g \otimes e^{ g^{-1}dg} \otimes \cdots$
 and simplify as follows:
 
 \begin{eqnarray*}
 && \hspace{-1cm} \sum_{r \geq 0}  \cdots \otimes x g \otimes (g^{-1}dg)^{\otimes r}\otimes g^{-1}y \otimes \cdots \\
 &=& \cdots \otimes  x g \otimes  g^{-1} y \otimes \cdots \quad + \quad\sum_{r \geq 1}  \cdots \otimes x g \otimes (g^{-1}dg)^{\otimes r}\otimes g^{-1}y \otimes \cdots \\
 &=&  \cdots \otimes  x g \otimes  g^{-1} y \otimes \cdots \quad +\quad \sum_{r \geq 1}  \cdots \otimes x g \otimes (g^{-1}dg)^{\otimes r}\otimes d(g^{-1})\otimes y \otimes \cdots  \quad  \\
 && + \quad \sum_{r \geq 1}  \cdots \otimes x g \otimes (g^{-1}dg)^{\otimes r-1} \otimes g^{-1}dg g^{-1}\otimes y \otimes \cdots \\
&=& \cdots \otimes  x g \otimes  g^{-1} y \otimes \cdots \quad -\quad  \cdots x g \otimes d(g^{-1})\otimes y \otimes \cdots\\
& = & \cdots \otimes x \otimes y \otimes \cdots  
     \end{eqnarray*}
   
   This completes the proof of Equation (\ref{middle}) which clearly implies Equation (\ref{N'}). Using a telescopic sum as above one can also prove that the following identities hold in $\N'(A\otimes K)$:
   
   \begin{equation}\label{left}
   1 \otimes e^{g^{-1}dg}\otimes g^{-1}x \otimes \cdots \quad = \quad (1 \otimes g^{-1}x \otimes \cdots) \quad - \quad (1 \otimes d(g^{-1})\otimes x \otimes \cdots)
   \end{equation}
   
   \begin{equation}\label{right}
   \cdots \otimes y g \otimes e^{g^{-1}dg} \quad = \quad (\cdots \otimes yg) \quad + \quad (\cdots \otimes y \otimes dg)
   \end{equation}
   
   We can now use Equations (\ref{N'}),(\ref{left}) and (\ref{right}) together we Lemma \ref{conjugation} to prove Equation (\ref{fixed r}) as follows:
   \begin{eqnarray*}
 &&\hspace{-1cm} {\phi}_K(1 \otimes e^{g^{-1}dg}\otimes (g^{-1}\beta g \otimes e^{g^{-1}dg})^{\otimes r})\\  
  &=&\phi_K(1 \otimes e^{g^{-1}dg}\otimes g^{-1}\beta\otimes\beta^{\otimes r-2}\otimes \beta g \otimes  e^{g^{-1}dg} )\\  
 &=&\phi_K(1 \otimes e^{g^{-1}dg}\otimes g^{-1}\beta\otimes\beta^{\otimes r-2}\otimes \beta g )
 +\phi_K(1 \otimes e^{g^{-1}dg}\otimes g^{-1}\beta\otimes\beta^{\otimes r-2}\otimes \beta  \otimes  dg )\\     
&=& \phi_K( 1 \otimes e^{g^{-1}dg}\otimes g^{-1}\beta\otimes\beta^{\otimes r-1}) g \\
&=& \left(\phi_K(1 \otimes  g^{-1}\beta\otimes\beta^{\otimes r-1}) -\phi_K(1 \otimes d(g^{-1})\otimes \beta^{\otimes r})\right)g \\
&=& g^{-1} \phi_K( 1\otimes \beta^{r}) g.
 \end{eqnarray*}
\end{proof}

\begin{corollary}\label{corollary: gauge-invariance}
Let $A,B$ be graded commutative dg algebras and
$$ \varphi: \hat{\barC}(A)\to B$$
 a morphism of filtered dg algebras whose kernel contains $\Deg(A)$.
Given any graded algebra $K$, we consider the linear map
  \[
 \tilde{\varphi}_K: (A\otimes K)^1  \to  B\hat{\otimes} K, \qquad \beta \mapsto \varphi_K(1\otimes e^{ \beta}),
\]
 where $\varphi_K:= (\varphi \otimes \id)\circ \phi_K$.
 The map $\tilde{\varphi}_K$ is equivariant with respect to $U_0(A\otimes K)$,
 the group of invertible elements of $A^0\otimes K^0$,
 which acts on $B\hat{\otimes} K$ by conjugation via the inclusion:
 $$A\otimes K \hookrightarrow \barC A\otimes K \to B\otimes K. $$
\end{corollary}

\begin{proof}
This is a direct consequence of Proposition \ref{proposition: gauge-invariance}.\end{proof}

\section{Parallel transport and automorphisms over $\loop X$}\label{section: automorphisms}

We are interested in describing $\infty$-local systems over $X$ in terms of the free loop space $\loop X$. In the case of ordinary local systems, there is the following well-known result, which appears for instance in Tadler, Wilson and Zeinalian \cite{TWZ}.

\begin{lemma}\label{basic2}
 $(E,\nabla)$ be a vector bundle with a flat connection and $\pi: \loop X \to X$ be the evaluation at $1\in S^1$.
Then $\pi^*E$  is a vector bundle with flat connection $\pi^*\nabla$. Moreover, the holonomies of $(E,\nabla)$ define an 
automorphism of the pair $(\pi^*E,\pi^*\nabla)$.
\end{lemma}

We will show that Lemma \ref{basic2} generalizes to the case of flat $\mathbb{Z}$-graded connections.

We treat the loop space $\loop X$ as a diffeological space, following ideas going back to the seminal work of Chen \cite{Chen}.
The interested reader is referred to the book \cite{diffeology} for a careful introduction to the theory of diffeological structures.

\subsection{Holonomy as an automorphism over $\loop X$}\label{subsection: automorphisms - automorphisms}

\begin{definition}
The iterated integral map 
$$\sigma: \barC \Omega^\bullet(X) \rightarrow \Omega^\bullet(\loop X)$$ is defined as follows.
Let $\ev: \Delta_k \times \loop X \rightarrow X^{\times (k+1)}$ be the  evaluation map:
\[ \ev(t_1,\cdots, t_k, \gamma):= (\gamma(1),\gamma(e^{2\pi i t_1}), \dots, \gamma(e^{2 \pi i t_k})).\]
The map $\sigma$ is the composition
\[ \xymatrix{
\Omega^\bullet(X)\otimes (\s \Omega^\bullet(X))^{\otimes k} \cong \Omega^\bullet(X)^{\otimes k+1} \ar[r]^(0.7){\iota}&  \Omega^\bullet(X^{\times (k+1)}) \ar[r]^{\ev^*}& \Omega^\bullet(\Delta_k \times \loop X) \ar[r]^(0.6){\int_{\Delta_k}}& \Omega^\bullet(\loop X).}\]
\end{definition}

\begin{theorem}[Chen \cite{Chen}]\label{theorem: Chen1}The  iterated integral map
\[ \sigma: \barC \Omega^\bullet(X) \rightarrow \Omega^\bullet(\loop X)\]
is a morphism of  dg algebras. 
If $X$ is connected, the kernel of $\sigma$ equals the subcomplex $\Deg(\Omega^\bullet(X))$. 
Moreover, if $X$ is simply-connected then $\sigma$ induces an isomorphism in cohomology:
\[ \barH(\Omega^\bullet(X)) \cong H^\bullet(\loop X).\]
\end{theorem}

\begin{remark}
In the following, we will always assume that $X$ is connected and that the iterated integral map $\sigma$
denotes the induced map on the quotient
$$ \sigma: \N(\Omega^\bullet(X)) = \barC \Omega^\bullet(X)/\Deg(\Omega^\bullet(X)) \to \Omega^\bullet(\loop X).$$
\end{remark}

Our aim is to define generalized holonomies by applying the results of \S \ref{subsection: formal holonomies - bar} and \S \ref{subsection: formal holonomies - normalized bar} to $A=\Omega^\bullet(X)$ and $K = \End V$, with $V$ some complex of vector spaces.
We first prove some auxiliary results.

\begin{lemma}\label{lemma: key}
Let $W$ be an $m$-dimensional compact manifold with boundary. If $\gamma: S^1 \rightarrow W^{\circ}$ is a smooth map, there exists an open set $U\subset W$ which contains $\gamma(S^1)$ and is homotopy equivalent to a finite $CW$-complex of dimension $1$.
\end{lemma}

\begin{proof}
Let $T$ be a smooth triangulation of $W$ and $D\subset W$ be the $m-2$ skeleton of the dual decomposition of $T$. 
We recall that $D$ is defined as follows: given a $k$-simplex $\Delta$ in $T$, then
\[ \Delta\cap D:= \bigcup_{|\sigma|=2} D(\sigma),\]
where the union runs over all two-dimensional faces $\sigma$ of $\Delta$ and $D(\sigma)$ is defined to be the convex hull
of the baricenters of all faces of $\Delta$ that contain $\sigma$. 

It is well known that $U=W\setminus D$ deformation retracts onto the $1$-skeleton of the triangulation $T$. For completeness, let us reproduce the proof of this fact. For each $k\geq 0$ denote by
$T_k$ the $k$-skeleton of the triangulation $T$. Clearly, $T_1 \cap U = T_1$ and therefore it is enough to prove that if $k>1$ then $T_k \cap U $ deformation retracts onto
$T_{k-1} \cap U$. Consider a $k$-dimensional simplex $\Delta$ of $T$ and let $b \in \Delta$ be its baricenter. We need to prove that $\Delta \cap U$ deformation retracts onto $ \partial (\Delta) \cap U$. Since $k>1$ we have $b \in D$. Any point in $v \in \Delta \setminus \{b\}$ can be written uniquely as a convex combination $ v= t b+ (1-t) w$ where $w \in \partial \Delta$ and this gives a deformation retraction of $\Delta \setminus \{b\}$ onto $\partial \Delta$. We claim that this retraction restricts to a deformation retraction of $ \Delta \cap U $ onto $\partial (\Delta) \cap U$. We need to prove that for any point $ v= t b+ (1-t) w$ which lies in $U$, the segment from $v$ to $w$ lies in $U$. Suppose this were not the case and choose $z \in D$ in the segment from $v$ to $w$. By definition $z \in D(\sigma)$ for some two dimensional face $\sigma$ of $\Delta$. We know that $D(\sigma) \cap \Delta$ is convex and also $b \in D(\sigma)$. Since $v$ is in the segment from $z$ to $b$ we conclude that $v \in D$. This contradicts the assumption.

For each $m$-simplex $\Delta$ of $T$, $D\cap \Delta$ is the union of finitely many compact submanifolds of dimension $m-2$. By general position, there exists an ambient isotopy $\Gamma$ that takes $D$  to $D'$ and such that $D' \cap \gamma(S^1)=\emptyset$. Therefore, $U':= W\setminus D'$ satisfies the conditions of the lemma.
\end{proof}

\begin{lemma}\label{lemma: skeleton}
Let $X$ be a smooth manifold without boundary and $\gamma:S^1 \rightarrow X$ be a smooth map. Then there exists an open subset $U \subset X$ which contains $\gamma(S^1)$ and which is homotopy equivalent to a CW-complex of dimension $1$.
\end{lemma}

\begin{proof}
Let $f: X \rightarrow \mathbb{R}$ be a proper Morse function and let $a_k, b_k \in \mathbb{R}$ be regular values of $f$ such that $\lim_{k \to \infty} a_k=- \infty$ and $\lim_{k \to \infty} b_k=+\infty$. Since the image of $\gamma$ is compact, there exists $n>>0$ such that $ \gamma(S^1) \subset f^{-1}(]a_n,b_n[)$. We consider \[W:=f^{-1}([a_n,b_n]),\] which is a compact manifold with boundary which contains $\gamma(S^1)$ in its interior. 
By Lemma \ref{lemma: key} there exists an open neighbourhood $\tilde{U}\subset W$ which contains $\gamma(S^1)$ and is homotopy equivalent to a CW-complex of dimension $1$. We claim that
\[ U:= \tilde{U} \cap f^{-1}((a_n,b_n))\]
satisfies the conditions of the lemma. By construction $\gamma(S^1) \subset U$, which is open in $X$.
Also, $\tilde{U}$ is a manifold with boundary with interior $U$ and therefore the inclusion $U \hookrightarrow \tilde{U}$ is a homotopy equivalence. We conclude that $U$ is homotopy equivalent to a CW-complex of dimension $1$. 
\end{proof}

\begin{remark}
By the same arguments one sees that the image of a smooth map $f: M\to X$, with $M$ a compact manifold of dimension $m$,
is contained in an open neighborhood $U\subset X$ that is homotopy equivalent to a CW-complex of dimension $m$.
\end{remark}

\begin{corollary}\label{corollary: trivialon1}
The image of every smooth loop $\gamma: S^1\to X$ is contained in an open subset $U\subset X$ which has
the property that the restriction of an arbitrary complex vector bundle $E\to X$ to $U$ is trivializable.
\end{corollary}

\begin{proof}
By Lemma \ref{lemma: skeleton}, the image of $\gamma$ is contained in an open subset $U\subset X$
which is homotopy equivalent to a finite one-dimensional CW-complex.
Since $GL_n(\mathbb{C})$ is connected, any complex vector bundle over a finite, one-dimensional CW-complex is trivializable.
\end{proof}

We can now state the main result of this section:

\begin{theorem}\label{theorem: automorphism}
Let $E\rightarrow X$ be a flat graded vector bundle and denote by $\pi: \loop X\rightarrow X$ the natural projection given by evaluation at $1\in S^1$. Then $\pi^*E$ is a flat graded vector bundle over $\loop X$ which has a distinguished automorphism $\hol(E) \in \Aut(\pi^*E)$, characterized by the property that
if $U \subseteq X$  is an open subset over which $E$ trivializes, i.e.
$E\vert_U \cong U\times V$, we have
\[ \hol(E)|_{\loop U} = \sigma_{\End V}(1\otimes  e^{\alpha}),\]
where $d_\nabla=d +\alpha$.
Moreover, if $f: Y \rightarrow X$ is a smooth map then
\[ \hol(f^*E) =(\loop f)^*\hol(E),\]
with $\loop f: \loop Y \to \loop X$ given by $\loop f(\gamma):=f\circ \gamma$.
\end{theorem}

\begin{proof}
Let us first consider the case of a trivial graded vector bundle $M\times V$ with flat $\mathbb{Z}$-graded connection given by a form $\alpha$.
By definition, we have
\[ \hol(E):= \sigma_{\End V}(1\otimes e^\alpha).\]
We notice that, strictly speaking, the right hand side of the above definition is a series of differential forms
on $\loop X$.
Its convergence is discussed carefully in \cite{Igusa}.
From the point of view of the diffeological structure on $\loop X$,
any plot $f: U\to \loop X$ yields a series of differential forms
$$f^*\sigma_{\End V}(1\otimes \alpha) + f^*\sigma_{\End V}(1\otimes \alpha^{\otimes 2}) + \cdots$$ 
on $U$ with values in $\End V$.
Inspection of the defining formulas of $f^*\sigma_{\End V}(1\otimes e^\alpha)$ shows that this series converges in the $\mathcal{C}^\infty$ Whitney-topology. One can define $f^*\sigma_{\End V}(1\otimes e^\alpha)$ as the limit and check that one thus obtains a differential form
on $\loop X$ in the diffeological sense.

In view of Lemma \ref{lemma: MC}, Corollary \ref{corollary: formal holonomies} and Theorem \ref{theorem: Chen1}, we know that $ \hol(E)$ satisfies
\[ d \hol(E)+[\pi^*\alpha, \hol(E)]=0,\]
and therefore defines an endomorphism of the flat bundle $(\pi^*E,d+\pi^*\alpha)$. In order to check that $\hol(E)$  is invertible, it suffices to check that the $0$-form component is so. This is true because it is the ordinary holonomy of the underlying connection. 
Notice that Corollary \ref{corollary: gauge-invariance} and Theorem \ref{theorem: Chen1} imply that the definition of $\hol(E)$ for trivializable graded bundles is independent of the chosen trivialization.
Corollary \ref{corollary: trivialon1}, implies that $\loop X$ can be covered by open subsets of the form $\loop U$ where 
$U\subseteq X$ is an open such that $E\vert_U$ trivializes. Since $\hol(E)$ is well defined over each open $\loop U$, and these definitions agree on the intersections $\loop U\cap \loop U' = \loop (U\cap U')$,
 it is globally well defined.

The remaining claim follows from naturality of Chen's iterated integral map.
\end{proof}

\begin{remark}
\hspace{0cm}
\begin{itemize}
\item[(1)]
The naturality of the holonomy under pull backs along smooth maps
is not imposed in the definition, but follows from the fact
that it is specified on sufficiently small opens
of the loop space.
\item[(2)] We remark that $\hol(E)$ is well defined for
any graded vector bundle $E$ that comes equipped with a $\mathbb{Z}$-graded connection --
i.e. the local defining formulas make perfect sense regardless of whether or not
the $\mathbb{Z}$-graded connection is flat.
The flatness of the $\mathbb{Z}$-graded connection
only enters in the verification that $\hol(E)$
is constant with respect to the pull back connection on $\pi^*E$.
\end{itemize}
\end{remark}

\subsection{Inverting the holonomy $=$ reversing the loops}\label{subsection: inverting}

For later use it will be helpful to work out a more explicit description
of the inverse to $\hol(E)$. As in the case of ordinary holonomies,
the inverse is given by considering the holonomy of the reversed loop.
We denote the map which reverses a loop by $r$, i.e.
$$ r: \loop X\to \loop X, \quad r(\gamma)(e^{i t}):=\gamma(e^{-it}).$$

\begin{remark}\label{remark: factorization}
Let us consider the space of pairs of composable loops, i.e.
pairs of loops with the same base point. It is given by the fibre product
$\loop X \times_X \loop X$ and is the domain of three interesting maps:
$$
\xymatrix{
& \loop X \times_X \loop X \ar[rr]^c \ar[dr]_{\mathrm{pr}_2} \ar[dl]^{\mathrm{pr}_1}&& \loop X\\
 \loop X && \loop X. &
}
$$
Here $c$ denotes the concatenation map\footnote{Strictly speaking, concatenation might lead us out of the space of smooth loops. One way to circumvent this problem is to define $\loop X$ as the space of loops which are constant in a neighborhood of $1\in S^1$.}
$$ c(\gamma,\gamma') := \begin{cases}
\gamma(2t) & \textrm{for } t\le \frac{1}{2},\\
\gamma'(2t-1) & \textrm{for } t\ge \frac{1}{2}.
\end{cases}$$
 and
$\mathrm{pr}_1$ and $\mathrm{pr}_2$ are the projection maps.

In the following, we will rely on two important properties of
iterated integrals with regard to these maps, which were established by Chen in \cite{Chen}, see also the exposition in \cite{AZ}:
\begin{enumerate}
\item For all $\alpha_1, \dots, \alpha_k \in \Omega^\bullet(X)$,
the following identity holds in $\Omega^\bullet(\loop X \times_X \loop X)$:
$$ c^*\sigma(1\otimes \alpha_1\otimes \cdots \otimes \alpha_k) = \sum_{i=0}^k\mathrm{pr}_1^*\sigma(1\otimes \alpha_1\otimes \cdots \otimes \alpha_i) \wedge \mathrm{pr}_2^*\sigma(1\otimes \alpha_{i+1}\otimes \cdots \otimes \alpha_k).$$
\item  For all $\alpha_1,\dots,\alpha_k \in \Omega^\bullet(X)$,
the following identity holds in $\Omega^\bullet(\loop X)$:
$$ (c \circ(\mathrm{id}\times r))^*\sigma(1\otimes \alpha_1\otimes \cdots \otimes \cdots \alpha_k)=0.$$
\end{enumerate}

\end{remark}

For the holonomy automorphism, these properties imply the following result:

\begin{proposition}\label{proposition: factorization}
Let $E$ be a graded vector bundle over $X$, which comes equipped with a $\mathbb{Z}$-graded
connection.
\begin{enumerate}
\item The holonomy $\hol(E)$ of $E$ satisfies the following factorization property:
$$ c^*\hol(E) = \mathrm{pr}_1^*\hol(E)\wedge \mathrm{pr}_2^*\hol(E).$$
\item The pull back of $\hol(E)$ along the reversal map $r$ is 
the inverse of $\hol(E)$, i.e.
$$ \hol(E) \wedge r^*\hol(E) = \mathrm{id} = r^*\hol(E) \wedge \hol(E).$$
\end{enumerate}
\end{proposition}

\begin{proof}
The proof of the first claim proceeds by reduction to the local case.
Locally, $\hol(E)$ is given by $\sigma_{\End(V)}(1\otimes e^\alpha)$
for some differential form $\alpha$ with values in the endomorphism algebra of a graded vector space. The factorization then follows from 
part 1. of Remark \ref{remark: factorization}.
In the same fashion, one obtains the second claim from part 2. of the same remark.
\end{proof}

\subsection{Automorphism of the pull back functor}\label{subsection: automorphisms - functor version}

In \S\ref{subsection: automorphisms - automorphisms} we constructed an automorphism $\hol(E)\in \Aut(\pi^*E,\pi^*D)$ in the dg category $\FlatZ(\loop X)$ of flat graded vector bundles over $\loop X$. This automorphism is given by the higher holonomies of the flat $\mathbb{Z}$-graded vector bundle $E$. We will now show that this construction lifts to the categorical level, i.e. we will construct a natural automorphism of the pull back functor
$\pi^*: \FlatZ(X) \to \FlatZ(\loop X)$.

\begin{remark}
Pull back along any smooth map $f: Y\to X$ induces a dg functor
$$f^*: \FlatZ(X) \to \FlatZ(Y).$$
In particular, the evaluation map 
$\pi: \loop X \to X, \quad \gamma \mapsto \gamma(1)$
yields a dg functor
$$ \pi^*: \FlatZ(X) \to \FlatZ(\loop X).$$
\end{remark}

\begin{definition}\label{definition: Ainfty-transformations}
Let $F$, $G: \C \to \D$ be dg functors between dg categories.
An $\Ainfty$-transformation $\mu: F \Longrightarrow G$ is the following data:
\begin{itemize}
\item An assignment $\rho_0: \Ob \C \to \mathsf{Z}^0\Hom(F(-),G(-))$ and
\item A collection of linear maps $\rho_r: \s \Hom(A_{r-1},A_r)\otimes \cdots \otimes \s \Hom(A_0,A_1) \to \Hom(F(A_{0}),G(A_{r}))$ of degree $0$ for $r>0$,
\end{itemize}
such that for all composable chains of morphisms $(f_r,\dots,f_1)$ in $\C$ the relation
\begin{eqnarray*}
G(f_r) \circ \rho_{r-1}(f_{r-1},\dots,f_{1}) - (-1)^{\#} \rho_{r-1}(f_r,\dots,f_{2}) \circ F(f_1)
=
\rho_\bullet b(f_r,\dots, f_1) - d \rho_r(f_r,\dots, f_1)
\end{eqnarray*}
is satisfied where $\#$ is $|f_2|+\dots +|f_r|-r+1$.
\end{definition}

\begin{remark}\label{remarkfunctor}
Given two dg categories $\C$ and $\D$, one
can form a category
$\dgFun_{\Ainfty}(\C,\D)$ whose objects are the dg functors form $\C$ to $\D$
and whose morphisms are $\Ainfty$-transformations. The composition
of two $\Ainfty$-transformations $\rho: F \Longrightarrow G$ and $\lambda: G\Longrightarrow H$
is given by 
$$ (\lambda \circ \rho)_r := \sum_{i=0}^r \lambda_i \circ \rho_{r-i}.$$
An $\Ainfty$-transformation $\rho$ from $F$ to itself is called an $\Ainfty$-automorphism
if it is invertible in $\dgFun_{\Ainfty}(\C,\D)$.

\end{remark}

\begin{theorem}\label{theorem: automorphism - categorical}
Let $X$ be a manifold and 
$\pi^*: \FlatZ(X) \to \FlatZ(\loop X)$ be the dg functor induced by pull back along the evaluation at $1\in S^1$. There is an $\Ainfty$-automorphism
$$\underline{\hol}: \pi^* \Longrightarrow  \pi^*,$$
characterized by the property that
if $\iota :U \hookrightarrow  X$  is the inclusion of an open set on which every vector bundle trivializes, then the $\Ainfty$-transformation:
\[ (\loop \iota)^* \circ \underline{\hol}: (\loop \iota)^* \circ \pi^* \Longrightarrow (\loop \iota)^* \circ \pi^*,\]
is described as follows with respect to trivializations:
\[( \loop \iota)^* \circ \underline{\hol}(E(0)) \cong \sigma_{\End V(0)}(1\otimes  e^{\alpha_0}),\]
\[(\loop \iota)^* \circ \underline{\hol}(\phi_r,\dots,\phi_1)\cong \sigma_{\End{W}}(1\otimes e^{ \alpha_r} \otimes \beta_r \otimes e^{ \alpha_{r-1}} \otimes \cdots \otimes e^{ \alpha_1} \otimes  \beta_1 \otimes e^{ \alpha_0}).\] 

Here $ E(0),\dots ,E(r)$ are objects in $ \FlatZ(X)$, $\phi_i\in \Hom(E(i-1),E(i))$ and we have chosen
trivializations $ \iota^* (E(i))\cong U \times V(i)$ such that $ D(i):= d + \alpha_i$ and $\phi_i=\beta_i$ where $\alpha_i \in \Omega^\bullet(U) \otimes \End(V(i)) $, $\beta_i \in \Omega^\bullet(U)\otimes \Hom(V(i-1), V(i))$ and $W=\bigoplus_i V(i)$.

\end{theorem}

\begin{proof}
In view of Corollary \ref{corollary: trivialon1} we know that there is at most one $\Ainfty$-transformation
which has a local expression as above. Let us check locally that the expression above satisfies
the conditions of an $\Ainfty$-transformation. The fact that 
\[ \sigma_{\End V(0)}(1\otimes  e^{\alpha_0}),\]
provides  an automorphism of $\pi^*E$ is guaranteed by Theorem \ref{theorem: automorphism}.
In order to prove that the equations for an $\Ainfty$-transformations are satisfied, we consider the differential form 
 \[ \gamma:= \alpha_0 + \cdots +\alpha_r+   \beta_1\tau_1+\cdots+ \beta_r\tau_r \in \Omega^{\bullet}(U) \otimes \End W\otimes F_\tau,\]
where $\tau_1, \dots, \tau_r$ are formal variables of degree $1-|\beta_i|$ and $F_\tau$ is the free
graded (but not graded commutative!) algebra on these variables.
By Corollary  \ref{corollary: formal holonomies}, we know that:
\[ [(\pi^* \otimes \mathrm{id})(\gamma),\sigma_{\End W\otimes F_\tau}(1\otimes e^{\gamma})]=\sigma_{\End W\otimes F_\tau} b(1\otimes e^{ \gamma}) -d\sigma_{\End W\otimes F_\tau}(1\otimes e^{\gamma}), \]
which is an equation in $\Omega^\bullet(\loop X)\otimes \End W\otimes F_\tau$.
As the components of this equation in front of $\tau_r \tau_{r-1}\cdots \tau_1 \in F_\tau$ we find:
\[\pi^*(\phi_r) \circ \underline{\hol}_{r-1}(\phi_{r-1},\dots,\phi_{1}) - (-1)^{|\phi_2|+\cdots +|\phi_r|-r+1}\underline{\hol}_{r-1}(\phi_r,\dots,\phi_{2}) \circ \pi^*(\phi_1)
= [\underline{\hol}_r,d](\phi_r,\dots, \phi_1).\]

This is precisely the condition required from an $\Ainfty$-transformation.
Corollary \ref{corollary: gauge-invariance} implies that the expression above is gauge invariant and therefore the local transformations glue to a global one.

It remains to prove that $\underline{\hol}$ is invertible.
We define a new $\A_\infty$-transformation
$r^*\underline{\hol}$ which is the composition of
$\underline{\hol}$ and the pull back along the reversal map
$r:\loop X \to \loop X$.
The fact that $r^*\underline{\hol}$ is an $\A_\infty$-transformation
is easy to check.
The claim that $r^* \underline{\hol}$ is inverse to $\underline{\hol}$ is a formal consequence of the fact that
the pull back $r^*\hol(E)$ along the reversal map
$r:\loop X\to \loop X$ is inverse to $\hol(E)$: if we apply this to the formal connection
$\gamma$ from above, we obtain the following equation in
$\Omega^\bullet(\loop X)\otimes \End W\otimes F_\tau$:
$$ \sigma_{\End W\otimes F_\tau}(1\otimes e^\gamma) \wedge r^*\sigma_{\End W\otimes F_\tau}(1\otimes e^\gamma) = 1.$$
If we extract the coefficient in front of $\tau_r \tau_{r-1}\cdots \tau_1$, this
yields
$$ \sum_{i=0}^r \underline{\hol}(\phi_r,\dots, \phi_{i+1}) \wedge r^*\underline{\hol}(\phi_i, \dots, \phi_1) = 0$$
provided $r>0$, which is the desired result.

\end{proof}

\section{Restriction to the based loop space}\label{section: based loops}

\subsection{Representations of the Pontryagin algebra}\label{subsection: based loops - Pontryagin product}

Let $* \hookrightarrow X$ be a pointed manifold.

\begin{definition}
The space of based loops $\basedloop X$ of $(X,*)$ is
the space of all smooth maps from $S^1\subset \mathbb{C}$ to $X$
that map a neighborhood of $1$ to the base point $*$.
\end{definition}

\begin{remark}
The diffeological space $\basedloop X$ comes equipped with 
the concatenation map $c$, which is the restriction of 
the concatenation map $c: \loop X\times_X \loop X \to \loop X$
which was considered in \S \ref{subsection: inverting}.
We notice that concatenation is not associative, but
this can be fixed by passing to Moore loops.
\end{remark}

\begin{definition}
The space of based Moore loops $\basedloop^\Moore X$ of $(X,*)$
is the space of pairs $(\gamma,l)$, where $l\in \mathbb{R}_{\ge 0}$ and
$\gamma: \mathbb{R}_{\geq 0}\rightarrow X$ is a smooth map which takes the value $*$ in a neighbourhood of $0$ and also in the interval $ (l-\epsilon,\infty) $ for some $\epsilon >0$. The based Moore loop space inherits the subspace topology and diffeology from $ C^{\infty}(\mathbb{R}_{\geq 0},X) \times \mathbb{R}$.
\end{definition}

\begin{remark}
\begin{enumerate}
\item
The space $\basedloop^\Moore X$ has the structure of a topological and diffeological monoid with concatenation product $c$ defined by
\[ c((\gamma, l),(\gamma',l')):= (\gamma \ast \gamma', l+l')\]
where
$$
(\gamma \ast \gamma')(t) \mapsto \begin{cases}
\gamma(t) & \textrm{ for } t\leq l\\
\gamma'(t-l) & \textrm{ for } t\geq l'. 
\end{cases}
$$
There is a natural  homotopy equivalence $q: \basedloop^\Moore X\rightarrow \basedloop X$ defined by:
\[q(\gamma,l)(e^{2 \pi i\theta}):= \gamma( l\theta), \text{ for } \theta \in [0,1].\]
This map is also compatible with the diffeological structure.
\item
Since  $\basedloop^\Moore X$ is a topological monoid,
its singular chain complex has the structure of a dg Hopf algebra. We will only be interested in the differential and the multiplication, which is given by the composition
$$
\xymatrix{
\chains(\basedloop^\Moore X)\otimes \chains(\basedloop^\Moore X) \ar[rr]^(0.55){\EZ}&&
\chains(\basedloop^\Moore X \times \basedloop^\Moore X) \ar[rr]^(0.6){c_*} &&
\chains(\basedloop^\Moore X),
}
$$
where $\EZ$ is the Eilenberg-Zilber map defined by
\begin{eqnarray*}
 \EZ(\mu \otimes \nu)(t_1,\cdots, t_{r+s}) &:=& \sum_{\chi \in \Sigma_{(r,s)}}(-1)^{|\chi|}(\EZ_\chi(\mu\otimes\nu))(t_1,\cdots,t_{r+s})\\
 &=& \sum_{\chi \in \Sigma_{(r,s)}} (-1)^{|\chi|}(\mu(t_{\chi(1)},\cdots,t_{\chi(r)}),\nu(t_{\chi(r+1)},\cdots,t_{\chi(r+s)})).
\end{eqnarray*}
We refer to this multiplication as the Pontryagin product.

\item The dg algebra $\chains(\basedloop^\Moore X)$ comes equipped with a natural augmentation
map $\epsilon$, which is induced by the map from
$\basedloop^\Moore X$ to a point.
\end{enumerate}
\end{remark}

\begin{definition}
The iterated integral map 
$$\sigma: \HC_\bullet (\Omega^\bullet(X)) \rightarrow \Omega^\bullet(\basedloop X)$$ is defined as follows.
Let $\ev: \Delta_k \times \basedloop X \rightarrow X^{\times k}$ be the  evaluation map:
\[ \ev(t_1,\cdots, t_k, \gamma):= (\gamma(e^{2\pi i t_1}), \cdots, \gamma(e^{2 \pi i t_k})).\]
The map $\sigma$ is the composition
\[ \xymatrix{
(\s\Omega^\bullet(X))^{\otimes k} \ar[r]^(0.55)\cong & \Omega^\bullet(X)^{\otimes k} \ar[r]^{\iota}&  \Omega^\bullet(X^{\times k}) \ar[r]^(0.4){\ev^*}& \Omega^\bullet(\Delta_k \times \basedloop X) \ar[r]^(0.6){\int_{\Delta_k}}& \Omega^\bullet(\basedloop X).}\]
\end{definition}

\begin{remark}
There is a natural map:
\[ \barC \Omega^\bullet(X) \rightarrow \HC_\bullet (\Omega (X)), \quad \alpha_0 \otimes \cdots \otimes \alpha_r \mapsto (\alpha_0\vert_*) \alpha_1 \otimes \cdots \otimes \alpha_r\]
which makes the following diagram commute
$$
\xymatrix{
\barC (\Omega^\bullet(X)) \ar[r]^\sigma \ar[d]& \Omega^\bullet( \loop X)\ar[d]^{\iota^*}\\
\HC_\bullet ( \Omega(X)) \ar[r]^\sigma& \Omega^\bullet(\basedloop X)}.$$
In particular, if $E$ is a trivial graded vector bundle with flat $\mathbb{Z}$-graded connection $D=d+\alpha$, then
the restriction $\hol_*(E)\in \Omega^\bullet( \basedloop X) \otimes \End E\vert_*$  of $\hol(E)$ to the based loop space is given by the formula
\[ \hol_*(E):=\sigma_{\End E\vert_*} ( e^\alpha). \]
\end{remark}

\begin{theorem}[Chen \cite{Chen}]\label{theorem: Chen2}
 The iterated integral map
\[ \sigma: \HC_\bullet(\Omega^\bullet(X)) \rightarrow \Omega^\bullet(\basedloop X)\]
is a morphism of  dg algebras. 
Moreover, if $X$ is simply-connected, then $\sigma$ induces an isomorphism
\[ \HH_{\bullet}(\Omega^\bullet(X)) \cong H^\bullet(\basedloop X).\]
\end{theorem}

Observe that the map $ q: \basedloop^\Moore X \rightarrow \basedloop X$ does not respect concatenation. Nevertheless,
the following holds:

\begin{lemma}\label{proposition: holonomies and Moore loops}
On the image of Chen's iterated integral map
$$\sigma: \HC_\bullet(\Omega^\bullet(X)) \to \Omega^\bullet(\basedloop X)$$
the maps $ c^*\circ q^*$ and  $(q\times q)^* \circ c^*$
coincide.
\end{lemma}

\begin{proof}
Out of two Moore loops
$(\gamma,l)$ and $(\gamma',l')$ in $X$, one can form the two
loops $q(c((\gamma,l),(\gamma',l')))$ and $c(q(\gamma,l),q(\gamma',l'))$.
We notice that these two loops are related by the following
piecewise linear bijection of the unit interval (we identify $S^1$ with $\mathbb{R}/\mathbb{Z})$:
$$
\tau(t):=\begin{cases}
t \frac{2l}{l+l'} & \quad \textrm{ for } t\le \frac{1}{2}\\
\frac{l}{l+l'} + (2t-1)\frac{l'}{l+l'} & \quad \textrm{ for } t \ge \frac{1}{2}
\end{cases},
$$
that is $$
q(c((\gamma,l),(\gamma',l')))\circ \tau = c(q(\gamma,l),q(\gamma',l')).$$

It is well-known that such reparametrization of $S^1$ leave iterated integrals unchanged. A discussion of this fact can for instance be found in \cite{AS}, \S 3 -- and in particular Definition 3.11, Lemma 3.12 and Remark 3.13 in loc.cit.
In the case at hand, this means that
for any plot $\hat{f}\times \hat{g}: U \to \basedloop^\Moore X \times \basedloop^\Moore X$, and any collection of differential forms
$\alpha_1,\dots,\alpha_k \in \Omega^\bullet(X)$, the equality
$$ (q \circ  c\circ \hat{f}\times \hat{g})^*\sigma(\alpha_1,\dots,\alpha_k)= (c \circ q\times q \circ \hat{f}\times \hat{g})^*\sigma(\alpha_1,\dots,\alpha_k)$$
holds.
This completes the proof.

\end{proof}

Suppose now that $E\to X$ is a flat graded vector bundle. Theorem \ref{theorem: automorphism} asserts that
the pull back of $(E,D)$ along the projection 
$$\pi: \loop X \to X, \quad \gamma \mapsto \gamma(1)$$
comes equipped with a natural automorphism $\hol(E) \in \Omega^\bullet(\loop X, \End E)$.
If we restrict this element to $\basedloop X$, the bundle $\pi^*E$ trivializes
and we obtain an element
$$ \hol_*(E) \in \Omega^\bullet(\basedloop X, \End E\vert_*)\cong \Omega^\bullet(\basedloop X) \otimes \End E\vert_*.$$

\begin{proposition}\label{proposition: holonomy on based loops}
Let $c$ be the concatenation map on $\basedloop^\Moore X$, $q: \basedloop^\Moore X \to \basedloop X$ the map defined above and $\mathrm{pr}_1, \mathrm{pr}_2: \basedloop^\Moore X\times \basedloop^\Moore  X\to \basedloop^\Moore X$  the natural projection maps. Then the identity
$$ c^*(q^*\hol_*(E)) = \mathrm{pr}_1^*(q^*\hol_*(E)) \wedge \mathrm{pr}_2^*(q^*\hol_*(E))$$
holds.
\end{proposition}

\begin{proof}
Recall that by Proposition \ref{proposition: factorization} in \S \ref{subsection: inverting}, $\hol(E)$ satisfies the factorization property
$$c^*\hol(E) = \mathrm{pr}_1^*\hol(E) \wedge \mathrm{pr}_2^*\hol(E).$$
This implies that the same factorization holds for the restriction
of $\hol(E)$ to $\basedloop X$.
Since by Lemma \ref{proposition: holonomy on based loops} pulling back along $q^*$ is compatible with the concatenation map
on the image of Chen's map $\sigma$, the claimed
factorization for $q^*\hol_*(E)$ follows.
\end{proof}

\begin{theorem}\label{theorem: reps of Pontryagin algebra from automorphisms}
The linear map
\begin{eqnarray*}
\widehat{\hol}_*(E): \chains(\basedloop^\Moore X) \to \End E\vert_*, \quad
\nu \mapsto (-1)^{\frac{(k+1)(k+2)}{2}}\int_{\Delta_k} \nu^*(q^*\hol_*(E))
\end{eqnarray*}
is a morphism of dg algebras. Here $k$ is the dimension of the simplex $\nu$ and $\End E\vert_*$ is seen as a dg algebra with differential $[\partial,-]$,
where $\partial$ is the restriction of the flat $\mathbb{Z}$-graded connection to the base point.
\end{theorem}

\begin{proof}
We denote the flat $\mathbb{Z}$-graded connection on $E$ by $D$.
The map $\widehat{\hol}_*(E)$ is compatible with the differentials
since $\hol(E) \in \Omega^\bullet(\loop X, \pi^*\End(E))$
is parallel with respect to $\pi^*D$ and the restriction
of $\pi^*D$ to the based loop space is $d_{DR} + \partial_*$.

Concerning the compatibility with the products, take two chains $\nu$ and $\mu$
on $\basedloop^\Moore X$. We claim that
$$ \widehat{\hol}_*(E)(c_*\EZ(\mu\otimes\nu)) = \widehat{\hol}_*(E)(\mu)\circ \widehat{\hol}_*(E)(\nu).$$
We evaluate the left-hand side:
\begin{eqnarray*}
\widehat{\hol}_*(E)(c_*\EZ(\mu \otimes \nu)) &=&  (-1)^{\frac{(r+s)(r+s+1)}{2}}
\int_{\Delta_{r+s}} (c\circ \EZ(\mu \otimes \nu))^*(q^*\hol_*(E))\\
& =&(-1)^{\frac{(r+s+1)(r+s+2)}{2}} \sum_{\chi \in \Sigma_{(r,s)}} (-1)^{|\chi|}\int_{\Delta_{r+s}} (\EZ_\chi(\mu \otimes \nu))^*c^*(q^*\hol_*(E))\\
&=& (-1)^{\frac{(r+s+1)(r+s+2)}{2}}\sum_{\chi \in \Sigma_{(r,s)}}(-1)^{|\chi|} \int_{\Delta_{r+s}} (\EZ_\chi(\mu \otimes \nu))^* (\mathrm{pr}_1^*q^*\hol_*(E)\wedge \mathrm{pr}_2^* q^* \hol_*(E))\\ 
&=&(-1)^{\frac{(r+s+1)(r+s+2)}{2}} \int_{\Delta_r\times \Delta_s} (\mu\times \nu)^*(\mathrm{pr}_1^*q^*\hol_*(E)\wedge \mathrm{pr}_2^*q^*\hol_*(E))\\
&=& (-1)^{\frac{(r+1)(r+2)}{2}+ \frac{(s+1)(s+2)}{2}} \left(\int_{\Delta_r}\mu^*(q^*\hol_*(E))\right)\circ \left(\int_{\Delta_s}\nu^*(q^*\hol_*(E))\right)\\
&=& \widehat{\hol}_*(E)(\mu)\circ \widehat{\hol}_*(E)(\nu).
\end{eqnarray*}
To pass from the second to the third line we applied Proposition \ref{proposition: holonomy on based loops}.
The transition from the third to the forth line uses the fact that $\EZ_\chi(\sigma,\mu)$
is the composition of
$$ \psi_\chi: \Delta_{r+s} \to \Delta_r\times \Delta_s, \quad (t_1,\cdots,t_{r+s})\mapsto ((t_{\chi(1)},\cdots,t_{\chi(r)}),(t_{\chi(r+1)},\cdots,t_{\chi(r+s)}))$$
with $\sigma\times \mu: \Delta_r\times \Delta_s \to \basedloop^\Moore X \times \basedloop^\Moore X$
and that
$$ \psi: \coprod_{\chi \in \Sigma_{(r,s)}}\Delta_{r+s} \to \Delta_r\times \Delta_s, \quad \psi = \coprod_{\chi \in \Sigma_{(r,s)}}\psi_{\chi}$$
is a diffeomorphism away from a subset of measure zero.
Notice that $\psi_\chi$ is orientation preserving if and only if $\chi$ is an even permutation.

Concerning the additional sign, we observe that in the passage from the fourth to the fifth
line, we have to separate the endomorphism-valued part from the rest of the expression.
In doing so, one picks up a Koszul sign $(-1)^{rs}$, which exactly accounts for the change
from $\frac{(r+s+1)(r+s+2)}{2}$ to $\frac{(r+1)(r+2)}{2}+ \frac{(s+1)(s+2)}{2}$.
\end{proof}

\begin{remark}
In Theorem \ref{theorem: reps of Pontryagin algebra from automorphisms}, the proof of the compatibility with the products
is independent of the flatness of the $\mathbb{Z}$-graded connection
on $E$.
The flatness enters only in the compatibility with the differentials.
\end{remark}

\subsection{The $\A_\infty$-functor}

The aim of this subsection is to extend the assignment $E\mapsto \widehat{\hol}_*(E)$ to an $\A_\infty$-funtor from the dg category $\FlatZ(X)$ of flat graded vector bundles over $X$ 
to the dg category $\Rep(\chains(\basedloop^\Moore X))$ of representations of $\chains(\basedloop^\Moore X)$.

Recall from \S \ref{subsection: automorphisms - functor version}
that the assignment $E \mapsto \hol(E)$ naturally extends to an automorphism
of the dg functor $\pi^*: \FlatZ(X)\to \FlatZ(\loop X)$. More concretely,
this amounts to the existence of certain linear maps
$$ \underline{\hol}_r: \s \Hom(E_{r-1},E_r)\otimes \cdots \otimes \Hom(E_0,E_1) \to \Hom(\pi^*E_0,\pi^*E_r))=\Omega^\bullet(\loop X,\Hom(\pi^*E_0,\pi^*E_r)),$$
which serve as a system of coherent homotopies for the failure of $\hol$ to be a natural transformation
from $\pi^*$ to itself. Upon restriction to the based loop space, one obtains a family of maps
into $\Omega^\bullet(\basedloop X)\otimes \Hom(E_0\vert_*,E_r\vert_*)$.

The following lemma, which can be proved by a simple computation, is useful for replacing $\Omega^\bullet(\basedloop X)$ by the Hochschild cochain complex $\s\HC^\bullet(\chains(\basedloop^\Moore X))$.

\begin{lemma}
\hspace{0cm}
\begin{enumerate}
\item Let $Y$ be a diffeological space and consider the space $\chains(Y)$ as a complex.
There is a morphism of complexes
$$ I^Y: \Omega^\bullet(Y) \to (\s \HC^\bullet(\chains(Y)), \tilde{b}_0),$$
given by
$$ I^Y(\alpha)(\s \mu_1\otimes\cdots \otimes \s \mu_k):=  \begin{cases}
(-1)^{\frac{(|\alpha|+1)(|\alpha|+2)}{2}} \int_{\mu_1} \alpha & \textrm{if } k=1\\
0 & \textrm{otherwise}.
\end{cases}
$$
Here $\tilde{b}_0$ is the coboundary operator induced
from $b_0$ under suspension.
\item If $V$ is a complex, we obtain a morphism of complexes
$$ I^Y_V: \Omega^\bullet(Y)\otimes V \to \s \HC^\bullet(\chains(Y),V)$$
given by the composition
$$
\xymatrix{
\Omega^\bullet(Y)\otimes V \ar[r]^(0.4){I^Y\otimes \id} & \s \HC^\bullet(\chains(Y))\otimes V \ar[r]^{\cong} & \s \HC(\chains(Y),V),
}$$
where the last morphism consists of rearrangements and Koszul-signs.
\item The map $I^Y_V$ is natural in $Y$ and $V$.
\end{enumerate}
\end{lemma}

\begin{lemma}\label{lemma: factorization}
Assume $Y$ is a diffeological monoid, with multiplication map $c$.
Equip $\chains(Y)$ with the corresponding multiplication $*$.
Moreover, let $B$ be a dg algebra and $\alpha \in \Omega^\bullet(Y)\otimes B$ a differential form on $Y$ that satisfies the factorization property
$$ c^*\alpha = \mathrm{pr}_1^*\alpha_1 \wedge \mathrm{pr}_2^*\alpha_2.$$

Then $I^Y_B(\alpha)$ satisfies
$$\tilde{b} I^{Y}_B(\alpha) + \tilde{\cup}(I^Y_B(\alpha_1)\otimes I^Y_B(\alpha_2))=0.$$
Here $\tilde{b}$ and $\tilde{\cup}$ denote the operators
corresponding to the Hochschild differential $b$ and the cup product
$\cup$ under suspension.

\end{lemma}

\begin{proof}
We only consider the case of trivial coefficients, the general case is analogous.
Let $\mu$ and $\nu$ be two chains on $Y$.
Computation leads to
\begin{eqnarray*}
\tilde{b}_1(I^Y(\alpha))(\s \mu\otimes \s \nu) &=& -(-1)^{|\alpha_2| + \frac{(|\alpha|+1)(|\alpha|+2)}{2}}\int_{\mu*\nu}\alpha,\\
\tilde{\cup}(I^Y(\alpha_1)\otimes I^Y(\alpha_2))(\s\mu\otimes \s \nu) &=& (-1)^{|\alpha_2|+|\alpha_1||\alpha_2|  +\frac{(|\alpha_1|+1)(|\alpha_1|+2)}{2}+\frac{(|\alpha_2|+1)(|\alpha_2|+2)}{2}} \left(\int_{\mu}\alpha_1\right)\left(\int_\nu \alpha_2\right),
\end{eqnarray*}
which sum to zero because 
$$\int_{\mu*\nu}\alpha = \left(\int_{\mu}\alpha_1\right)\left(\int_\nu \alpha_2\right),$$
as in the proof of Theorem \ref{theorem: reps of Pontryagin algebra from automorphisms}.

\end{proof}

\begin{definition}

For $r\ge 1$, we define linear maps
$$\varphi_r: \s\Hom(E_{r-1},E_r)\otimes \cdots \otimes \s \Hom(E_0,E_1)\to \s \HC^\bullet(\chains(\basedloop X),\Hom(E_0\vert_*,E_r\vert_*))$$
 of degree $0$ as the compositions $I^{\basedloop^\Moore X}_{\Hom(E_0\vert_*,E_r\vert_*)}\circ \underline{\hol}_r$.
\end{definition}

\begin{theorem}\label{theorem:restriction}
 The collection of multilinear maps $(\tilde{\varphi}_r)_{r\ge 1}$
defined
by 
$$\tilde{\varphi}_1(\s \phi):= \phi\vert_* +  \varphi_1(\s \phi), \qquad  \textrm{and} \qquad \tilde{\varphi}_r = \varphi_r \quad \textrm{for } r >1,$$
 yields an $\A_\infty$-functor from the dg category $\FlatZ(X)$ of flat $\mathbb{Z}$-graded connections
over $X$ to the dg category $\Rep(\chains(\basedloop^\Moore X))$ of representations of $\chains(\basedloop^\Moore X)$.
\end{theorem}

\begin{proof}
Let $(\phi_r,\dots,\phi_1)$ be a composable chain of morphisms in $\FlatZ(X)$. Checking that $(\tilde{\varphi}_r)_{r\ge 1}$ is an $\A_\infty$-functor amounts to check that
the various ways to apply the maps $\tilde{\varphi}_r$, the differentials and the compositions in the dg categories
are compatible, i.e. produce elements of $\s \HC^\bullet(\chains(\basedloop^\Moore X)),\Hom(E_0\vert_*,E_r\vert_*))$
that satisfy a linear relation.
By construction, all elements of $\s \HC^\bullet(\chains(\basedloop^\Moore X)),\Hom(E_0\vert_*,E_r\vert_*))$
that arise in these considerations
are linear maps with domain either $\s\chains(\basedloop^\Moore X)$
or $(\s \chains(\basedloop^\Moore X))^{\otimes 2}$.

Let us first focus on the component in $\s \Hom(\s \chains(\basedloop^\Moore X),\Hom(E_0\vert_*,E_r\vert_*)$,
i.e. the Hochschild cochains we can produce from $\phi_r,\dots,\phi_1$ which depend in a linear fashion on
$\s\chains(\basedloop^\Moore X)$.
The defining equation in this component reads
\begin{eqnarray*}
\tilde{\varphi}_\bullet(b(\phi_r\otimes \cdots \otimes \phi_1)) &=&\tilde{b}_0\tilde{\varphi}_r(\phi_r\otimes\cdots\otimes \phi_1)\\
&& +\tilde{\cup}(\s\phi_r\vert_* \otimes \tilde{\varphi}_{r-1}( \phi_{r-1}\otimes \cdots \otimes \phi_1)\\
&& + \tilde{\cup}(\tilde{\varphi}_{r-1}(\phi_r\otimes \cdots \otimes \phi_2) \otimes \s \phi_1\vert_*),
\end{eqnarray*}
where $\tilde{b}_0$ and $\tilde{\cup}$ are the operations on the shifted Hochschild cohomology complexes
induced from the terms in the Hochschild differential which do not increase the number of arguments, and the cup product, respectively.
It is straightforward to check that this equation is a consequence of the defining equations for $\underline{\hol}$ to
be an $\A_\infty$-transformation.

Concerning the component in $\s \Hom((\s \chains(\basedloop^\Moore X))^{\otimes 2},\Hom(E_0\vert_*,E_r\vert_*))$
the defining relation reads
$$
\tilde{b}_1\tilde{\varphi}_r(\phi_r\otimes\cdots\otimes \phi_1) + \sum_{k+l=r}\tilde{\cup}(\tilde{\varphi}_k(\phi_r\otimes\cdots\otimes \phi_{l+1}) \otimes \tilde{\varphi}_l(\phi_l\otimes \cdots \otimes \phi_1)) = 0.
$$

Observe that as a consequence of Proposition \ref{proposition: holonomies and Moore loops}, we have the following equality
of differential forms over $\basedloop^\Moore X \times \basedloop^\Moore X$ with values in $\Hom(E_0\vert_*,E_r\vert_*)$:
$$c^*\underline{\hol}_r(\phi_r\otimes\cdots\otimes \phi_1) = \sum_{k+l=r}\mathrm{pr}_1^*\underline{\hol}_{k}(\phi_r\otimes\cdots\otimes\phi_{l+1}) \wedge \mathrm{pr}_2^*\underline{\hol}_l(\phi_l\otimes\cdots\otimes \phi_1).$$
Hence we can apply Lemma \ref{lemma: factorization}
and obtain the remaining relations for $\tilde{\varphi}$ to be an $\A_\infty$-functor.
\end{proof}

\subsection{Reconstructing a flat $\mathbb{Z}$-graded connection} \label{subsection: based loops - reconstructing the connection}

The aim of this subsection is to prove that the $\Ainfty$-functor $\tilde{\varphi}: \FlatZ(X) \to \Rep(\chains(\basedloop^\Moore X))$
is an $\Ainfty$ quasi-equivalence of dg categories, i.e. to establish

\begin{theorem}\label{functor}
Let $X$ be a connected manifold.
 The $\Ainfty$-functor
 $$ \tilde{\varphi}: \FlatZ(X) \to \Rep(\chains(\basedloop^\Moore X))$$
 is a quasi-equivalence, i.e.
 \begin{enumerate}
 \item
 The chain maps given by $\tilde{\varphi}_1$ induce isomorphisms between the cohomologies
 of the Hom-complexes.
 \item The induced functor between the homotopy categories is an equivalence.
 \end{enumerate}
\end{theorem}

Let us start by proving the first claim of Theorem \ref{functor}
for the case of ordinary flat vector bundles.
It turns out that  this reduces to a result due to
F\'elix, Halperin and Thomas \cite{FHT}.

\begin{proposition}\label{proposition: FHT}
Let $X$ be a connected manifold. 
The restriction of $\tilde{\varphi}_1$ to the trivial coefficient system
$\mathbb{R}$ yields a chain map
$$ \varphi: \Omega^\bullet(X) \to  \HC^\bullet(\chains(\basedloop^\Moore X))$$
which induces an isomorphism in cohomology.
Furthermore, the result extends to the case of coefficients
with values in ordinary flat vector bundles $(E,\nabla)$,
and the module over $\chains(\basedloop^\Moore X)$
determined by the holonomy representation corresponding to $(E,\nabla)$.
Explicitly, the chain map $\varphi$ is given by
$$ \beta \mapsto \varphi_\beta(\s \mu_1 \otimes \cdots \otimes \s \mu_m):= 
\begin{cases}
1 \mapsto \beta \vert_* & \textrm{ if } m=0,\\
\s \mu \mapsto \pm\int_{\Delta_k \times [0,1]} \hat{\mu}^*\beta  & \textrm{ if } m=1,\\
0 & \textrm{ otherwise},
\end{cases}
 $$
 where $\hat{\mu}: \Delta_k\times [0,1] \to X$ is the map adjoint
 to $q\circ \mu: \Delta_k \to \basedloop^\Moore X \to \basedloop X$.
\end{proposition}

\begin{proof}
We first consider the case of trivial coefficients $\mathbb{R}$.
It is established in \cite{FHT}, Theorem 6.3, that one has a span of quasi-isomorphisms
of dg coalgebras
$$
\xymatrix{
 & \HC_\bullet(\chains(\basedloop^\Moore X), \chains(\mathsf{P}_*^{\Moore} X)) \ar[dl]_{r} \ar[dr]^{\ev_*}& \\
\HC_\bullet(\chains(\basedloop^\Moore X)) & & \chains(X).
}
$$
Here $\mathsf{P}_*^{\Moore}X$ denotes the space of Moore paths
which start at the based point $*$. The Moore loops based at $*$ act
on this space by pre-composition and thus equip
$\chains(\mathsf{P}_*^\Moore X)$ with a dg module structure over
$\chains(\basedloop^\Moore X)$.
The morphism $r$ is induced by the augmentation map
$\chains(\mathsf{P}_*^\Moore X) \to \mathbb{R}$,
while $\ev_*$ pushes chains on $\mathsf{P}^\Moore_*X$ forward to $X$ along
the evaluation map and maps everything else to zero.
Dualizing, one obtains that the dg algebras
$C^\bullet(X)$ and $\HC^\bullet(\chains(\basedloop^\Moore X))$ are quasi-isomorphic,
see Theorem 7.2 of \cite{FHT}.
We claim that the induced map in cohomology
$$ H^\bullet(X) \to \HH^\bullet(\chains(\basedloop^\Moore X))$$
coincides with $H^\bullet(\varphi)$ modulo (inessential) signs.

In order to verify this claim, we have to understand
how to find 
a cocycle $\tilde{x}$ of $HC_\bullet(\chains(\basedloop^\Moore X),\chains(\mathsf{P}_*^\Moore X))$
which maps under $r$ to a given cocycle $x$ of $\HC_\bullet(\chains(\basedloop^\Moore X))$.
One proceeds as follows: define $x_{(0)}$ to be the image
of $x$ under the inclusion $x\mapsto 1\otimes x \subset \HC_\bullet(\chains(\basedloop^\Moore X),\chains(\mathsf{P}_*^\Moore(X))$.
Since the inclusion is not a chain map,
$x_{(0)}$ fails to be a cocycle. The defect for closedness is given by the image of $x_{(0)}$ under the linear map
\begin{eqnarray*}
M: \HC_\bullet(\chains(\basedloop^\Moore X),\chains(\mathsf{P}_*^\Moore(X))
&\to& \HC_\bullet(\chains(\basedloop^\Moore X),\chains(\mathsf{P}_*^\Moore(X)),\\ \xi\otimes \s \mu_1\otimes \cdots\otimes \s \mu_k &\mapsto & (\xi\cdot \mu_1)\otimes \s \mu_2\otimes \cdots \otimes \s \mu_k.
\end{eqnarray*}
Recall that $P_*^\Moore X$ is contractible, hence we obtain
a homotopy operator $h$ on $\chains(\mathsf{P}_*^\Moore(X))$,
which we extend to the Hochschild chain complex.
If one defines
$$ x_{(1)} := -hM(x_{(0)}),$$
the sum $x_{(0)} + x_{(1)}$ is still not closed, but the failure
can be expressed as $-[M,h]M(x_{(0)})$.
Now consider the formal expression
$$ \tilde{x} := x_{(0)} + x_{(1)} + x_{(2)} + \cdots,$$
where for $k\ge 1$ we set $x_{(k)}:=(-1)^k h ([M,h])^{k-1}M(x_{(0)}).$
Since $M$ decreases the tensor degree, $\tilde{x}$ is well-defined.
Moreover, it is closed and maps to $x$ under $r$ by construction.
To finally obtain a singular chain on $X$,
we apply $\ev_*$ to $\tilde{x}$.

We claim that for 
$k\ge 2$, the smooth singular chains $\ev_*(x_{(k)})$ pair to zero with any differential form on $X$ and hence are cohomologically negligible. 
 As a consequence, the map
$$H(\ev_*)\circ H(r)^{-1}: \HH_\bullet(\chains(\basedloop^\Moore X))
\to H_\bullet(X)$$ can be represented by
$x\mapsto \ev_*(x_{(0)}+ x_{(1)})$. Using that for $\mu \in \chains(\mathsf{P}_*^\Moore X)$ and $\beta \in \Omega^\bullet(X)$
the identity
$$ \int_{\ev_*h(\mu)}\beta = \pm\int_{\mu} \sigma(\s \beta) \qquad (*)$$
holds,
one sees that the above map is dual
to $\varphi$ on the level of cohomology.

To prove that for $k\ge 2$ the chains $\ev_*(x_{(k)})$ pair to zero with every differential form, one uses $(*)$ and the fact that all pairings between
singular chains and differential forms on the path space that have the form
$$ \int_{h(\mu)} \sigma(\s \beta)$$
vanish. This is true because the
chain $\ev_* hh(\mu)$ on $X$, which is given by simplicial approximation
of 
$$\Delta_k \times [0,1]^{\times 2} \to X, \quad ((t_1,\dots,t_k),s,s') \mapsto \mu(t_1,\dots,t_k)(st), $$ 
has nowhere full rank. We observe that the whole proof remains valid if one replaces
the trivial coefficient system $\mathbb{R}$ by a flat vector bundle.
\end{proof}

\begin{definition}\label{definition: auxiliary filtration}
Let $E$ be a graded vector bundle over $X$.
The {\em auxiliary grading} on $\Omega^\bullet(X,E)$ and $\HC^{\bullet}(\chains(\basedloop^\Moore X),E\vert_*)$
is given by the following graded subspaces for $k\ge 0$: 
\begin{itemize}
\item $\Omega_{(k)}(X,E):= \Omega^{k}(X,E)$, i.e. the auxiliary grading equals the form-degree.
\item $\HC^{\bullet}_{(k)}(\chains(\basedloop^\Moore X),E\vert_*) = \Hom(\HC_{-k}(\chains(\basedloop^\Moore X)),E\vert_*)$, i.e. the auxiliary grading is dual to the total degree of $\HC_\bullet(\chains(\basedloop^\Moore X))$.\footnote{We remind the reader that we always work with cohomological gradings -- as a consequence the dg algebra
of singular chains is concentrated in non-positive degrees.}
\end{itemize}

The auxiliary gradings determine filtrations
$F_\bullet \Omega^\bullet(X,E)$ and $F_\bullet\HC^\bullet(\chains(\basedloop^\Moore X),E\vert_*)$:
\begin{itemize}
\item $F_r\Omega^\bullet(X,E):= \bigoplus_{k\ge r} \Omega^\bullet_{(k)}(X,E)$,
\item $F_r\HC^\bullet(\chains(\basedloop^\Moore X),E\vert_*) = \bigoplus_{k\ge r} \HC^\bullet_{(k)}(\chains(\basedloop^\Moore X),E\vert_*)$.
\end{itemize}
\end{definition}

\begin{remark}

The graded subspaces $F_r\Omega^\bullet(X,E)$ and $F_r\HC^\bullet(\chains(\basedloop^\Moore X),E\vert_*)$
are dg submodules of $\Omega^\bullet(X,E)$ and $\HC^\bullet(\chains(\basedloop^\Moore X),E\vert_*)$, respectively. The corresponding filtrations 
are descending and complete, i.e. the intersection of all the subspaces is zero.
Moreover they are locally finite in the sense that, if we fix $p\in \mathbb{Z}$, the intersections
of $F_r \Omega^\bullet(X,E)$ and $F_r\HC^\bullet(\chains(\basedloop^\Moore X),E\vert_*)$
with the subspaces of total degree $p$
are non-zero only for a finite number of $r\ge 0$.

\end{remark}

\begin{corollary}\label{corollary: FHT}
The $\Ainfty$-functor $\tilde{\varphi}$ induces isomorphisms between the homomorphism complexes, i.e.
for any graded vector bundles $E$, $F$ over $X$ with flat $\mathbb{Z}$-graded connection, the chain map
$$ \tilde{\varphi}_1: \Omega^\bullet(X,\Hom(E,F)) \to \HC^{\bullet}(\chains(\basedloop^\Moore X),\Hom(E\vert_*,F\vert_*))$$
induces an isomorphism on cohomology.
\end{corollary}

\begin{proof}
Let $F_\bullet \Omega^\bullet(X,\Hom(E,F))$ and $F_\bullet \HC^\bullet(\chains(\basedloop^\Moore X),\Hom(E\vert_*,F\vert_*))$
be the filtrations introduced in Definition \ref{definition: auxiliary filtration}.
The map $\tilde{\varphi}_1$ is compatible with the filtrations. Since the associated
spectral sequences $\mathcal{E}$ and $\mathcal{F}$ converge to the cohomologies of
$\Omega^\bullet(X,\Hom(E,F))$ and $\HC^\bullet(\chains(\basedloop^\Moore X),\Hom(E\vert_*,F\vert_*))$,
it suffices to check that $\tilde{\varphi}_1$ induces an isomorphism between the second pages of $\mathcal{E}$
and $\mathcal{F}$, respectively.

The first sheets are given by
$$\mathcal{E}_1=\Omega^\bullet(X,\Hom(H_\partial(E),H_\partial(F))) \quad \textrm{and} \quad\mathcal{F}_1=\HC^\bullet(\chains(\basedloop^\Moore X),\Hom(H_\partial(E\vert_*),H_\partial(F_*))),$$
where $H_\partial(E)$ and $H_\partial(F)$ refer to cohomology with respect to the fibrewise coboundary
operators $\partial$ on $E$ and $F$.
The induced coboundary operators on the second sheets are given by the covariant derivative with respect to the induced flat connection
$[\nabla]$ on the graded vector bundle $\Hom(H_\partial(E),H_\partial(F))$ and the Hochschild coboundary operator $b$, twisted by the holonomies
with respect to $[\nabla]$.
Moreover, the induced chain map coincides with $\varphi$ from Proposition \ref{proposition: FHT}
and hence induces an isomorphism between the second sheets $\mathcal{E}_2$ and $\mathcal{F}_2$, respectively.
\end{proof}

\begin{proposition}\label{prop: quasi-surjectivity}
The $\Ainfty$-functor $\tilde{\varphi}: \FlatZ(X) \to \Rep(\chains(\basedloop^\Moore X))$ is quasi-surjective,
i.e
for every object $\tau$ of $\Rep(\chains(\basedloop^\Moore X))$
there is a flat graded vector bundle $E$ such that
$\widehat{\hol}_*(E)$ is quasi-isomorphic to $\tau$.
\end{proposition}

\begin{proof}
Recall from \S \ref{subsection: reps of dg algebras} that every object
in $\Rep(\chains(\basedloop^\Moore X))$ is quasi-isomorphic to one where the underlying complex is of the form
$(V,0)$, i.e. has trivial coboundary operator. Let $\tau$ be such an object. By definition, $\tau$ is a Maurer-Cartan element of the dg algebra
$\HC^\bullet(\chains(\basedloop^\Moore X),\End V)$.
We decompose $\tau$ with respect to the auxiliary degree we introduced in Definition \ref{definition: auxiliary filtration}, i.e.
$$ \tau = \sum_{k \ge 0}\tau_{(k)}, \qquad \tau_{(k)}\in \Hom(\HC_{-k}(\chains(\basedloop^\Moore X)),\End V).$$
Since $\chains(\basedloop^\Moore X)$ is concentrated in non-positive degrees, the only component of
$\HC_k(\chains(\basedloop^\Moore X))$ in degree $0$ is $\mathbb{R}$. Since the restriction of
$\tau$ to $\mathbb{R}$ is required to vanish, we conclude that $\tau_{(0)}$ vanishes as well.
We next consider the component $\tau_{(1)}$. By definition, it is given by the restriction of $\tau_{(1)}$
to elements of $\HC_\bullet(\chains(\basedloop^\Moore X))$ of degree $-1$. This are exactly the chains
on $\basedloop^\Moore X$ of degree $0$, i.e. linear combinations of points in $\basedloop^\Moore X$.
The map
$$ \tau_{(1)}: C_0(\basedloop^\Moore X) \to \End V$$
is part of the linear component $\tau_1$ of an $\Ainfty$-morphism from
$\chains(\basedloop^\Moore X)$ to $\End V$.
Hence it induces a morphism of graded algebras on the level of homology. Since the coboundary operator on $V$
is trivial we obtain a morphism of graded algebras
$$ [\tau_{(1)}]: H_0(\basedloop^\Moore X) \to \End V.$$
We note that $H_0(\basedloop^\Moore X)$ is the group-ring of the fundamental group $\pi_1(X,*)$ of $X$.
By the universal property of the group-ring, $[\tau_{(1)}]$ corresponds
to a morphism of groups
$$ \rho_\tau: \pi_1(X,*) \to GL(V),$$
i.e. a representation of the fundamental group of $X$.
By the adjoint bundle construction, there is a vector bundle $E$, equipped with a flat connection $\nabla$,
such that the holonomy representation of $(E,\nabla)$ coincides with $\rho_\tau$. The $\mathbb{Z}$-grading on $V$
induces a $\mathbb{Z}$-grading on $E$, which is preserved by $\nabla$.
We now consider the restriction of the $\Ainfty$-functor $\tilde{\varphi}$
to the object $(E,\nabla)$. This gives us an $\Ainfty$-morphism
$$\tilde{\varphi}: \Omega^\bullet(X,\End E) \to \HC^\bullet(\chains(\basedloop^\Moore X),\End V).$$
The left-hand side is equipped with the covariant derivative corresponding to $\nabla$,
while the right-hand side is equipped with the Hochschild boundary operator twisted
by $\rho_\tau$.
We observe that
$$ \tau_+ := \sum_{k\ge 2} \tau_{(k)}$$
is a Maurer-Cartan element of the right-hand side.
By Proposition \ref{proposition: HPT} from Appendix \ref{appendix: HPT} we can find
a Maurer-Cartan element $\alpha = \sum_{k\ge 2} \alpha_{(k)}$ of
$\Omega^\bullet(X,\End E)$ such that
$\tilde{\varphi}_*(\omega)$ is gauge-equivalent to $\tau_+$. Finally, we observe that this is tantamount to saying that the $\Ainfty$-functor
$\tilde{\varphi}$ sends $E$, equipped with the the flat $\mathbb{Z}$-graded connection
corresponding to $(\nabla,\alpha_2,\alpha_3,\dots)$, to an object of
$\Rep(\chains(\basedloop^\Moore X))$ which is isomorphic to $\tau$.
This completes the proof.
\end{proof}

As a consequence we have the following result:

\begin{corollary}
Let $X$ be a connected manifold.
The dg categories $\FlatZ(X)$ and $\Rep(\chains(\basedloop^\Moore X))$
are quasi-equivalent.
\end{corollary}

\begin{remark}
One can obtain this corollary in a less direct way by combining
\begin{itemize}
\item The quasi-equivalence between $\FlatZ(X)$
and $\Rep^\infty(\pi_\infty(X))$, established by Block-Smith  \cite{BS}.
\item The quasi-equivalence between $\Rep^\infty(\pi_\infty(X))$
and $\Rep^\infty(\chains(\basedloop^\Moore X))$, established by Holstein  \cite{Holstein}.
\end{itemize}
\end{remark}

\appendix

\section*{Appendix}

\section{The Maurer-Cartan equation in a filtered dg algebra}\label{appendix: HPT}

Here we establish a technical result regarding Maurer Cartan elements
of dg algebras that are equipped with an auxiliary grading. 
This result is used in the proof of Proposition \ref{prop: quasi-surjectivity}.


\begin{definition}
A filtration on a dg algebra $A$ is a decreasing sequence of dg ideals
\[ A=F_0(A) \supseteq F_1(A) \supseteq F_2(A)\supseteq \cdots,\]
such that:
\[ \bigcap_{k\in \mathbb{N}} F_k(A)=0,\]
and
\[ F_i(A) \cdot F_j(A) \subseteq F_{i+j}(A).\]
The filtration is said to be complete if the natural inclusion
\[ A \hookrightarrow \hat{A}:=\varinjlim A / F_k(A)\]
is an isomorphism.
\end{definition}

\begin{definition}
Let $A,B$ be filtered algebras. We say that an $\A_\infty$-morphism $\psi: A \rightarrow B$
is filtration preserving if
\[ \psi_k(a_1,\cdots a_k)\in F_{l_1+ \cdots +l_k-1}(B) \textit{ whenever } a_i \in F_{l_i}(A).\]
We say that $\psi: A \rightarrow B$ is strongly filtration preserving if
\[ \psi_k(a_1,\cdots a_k)\in F_{l_1+ \cdots +l_k+k-1}(B) \textit{ whenever } a_i \in F_{l_i}(A).\]\end{definition}

\begin{remark}
If $A$ is a filtered algebra then each of the algebras $F_k(A)$ inherits a filtration given by
\[ F_i(F_k(A)):=F_{i+k}(A).\]
If $\psi: A \rightarrow B $ is a filtration preserving $\A_\infty$-morphism then for each $k\geq 1$, $\psi$ restricts to an $\A_\infty$-morphism:
\[ \psi: F_k(A) \rightarrow F_k(B),\]
which is strongly filtration preserving.
\end{remark}

\begin{remark}
A filtration on a dg algebra $A$ induces a Hausdorff topology on the underlying vector space of $\hat{A}$ where the basis of neighbourhoods for zero is $\{F_k(A)\}_{k \in \mathbb{N}}$. Let $A,B$ be dg algebras with filtration. If $\psi:A \rightarrow B$ is a strongly filtration preserving $\A_\infty$-morphism then the push-forward map
\[ \psi_*: \MC(\hat{A})\rightarrow \MC(\hat{B}), \quad  \psi_*(\alpha):= \sum_{k\geq 1}\psi_k(\alpha^{\otimes k}) \]
is well defined because the infinite sequence on the right is convergent.
\end{remark}\begin{definition}
Let $A$ be a dg algebra.
An auxiliary grading on $A$ is a decomposition
$$ A = \bigoplus_{k\ge 0} A_{(k)},$$
with $\{A_{(k)}\}_{k\ge 0}$ a collection of graded subspaces of $A$ such that:
\begin{itemize}
\item For fixed degree $r \in \mathbb{Z}$ we have
$A^r \cap A_{(k)}= 0 \textit{ for } k\gg0$.
\item The collection $\{A_{(k)}\}_{k\ge 0}$ is compatible with the differential and the multiplication
in the sense that
$$ dA_{(k)}\subseteq A_{(k+1)} \qquad \textrm{and} \qquad A_{(k)}\cdot A_{(l)} \subset A_{(k+l)}$$
hold.
\end{itemize}

\end{definition}
\begin{remark}

Let $A$ be a dg algebra with an auxiliary grading, then there is an induced filtration given by:
\[ F_k(A):= \bigoplus_{l\geq k} A_{(l)}.\]
Recall that $U(A)$ denotes the group of invertibles 
of degree $0$ in $A$.
Since the filtration is locally finite, i.e. in each degree $r$ only a finite number of intersections $A^r \cap F_kA$ are non-trivial,
the filtration is complete.

Denote by $\G(A)$ the subset of $U(A)$ given by elements of the form
$$ e^\xi = 1 + \xi + \xi \cdot \xi + \xi \cdot \xi \cdot \xi +...$$
for some $\xi \in F_1(A)^0$.
The gauge-action of $U(A)$ restricts to an action on the spaces $\MC(F_2(A))$. 
 The Maurer-Cartan equation for $\tau = \sum_{k\ge 2} \tau_{(k)} \in \MC(F_2(A))$ is equivalent to
$$ d \tau_{(k)} = -\sum_{r+s = k+1} \tau_{(r)} \cdot \tau_{(s)}$$
for all $k \ge 2$.
An element $e^\xi$ of $\G(A)$ acts on $\tau$ by
$$  \tau \bullet e^\xi = e^{-\xi}\tau e^{\xi} + d \xi.$$
In terms of the decomposition $\tau = \sum_{k \ge 2} \tau_{(k)}$
we have
$$ ( \tau \bullet e^\xi)_{(k)} = d(\xi_{(k-1)})+\sum_{i_1+\cdots +i_r+j+l_1+\cdots+l_s=k} (-1)^{t} \xi_{(i_1)}\cdots \xi_{(i_r)} \tau_{(j)} \xi_{(l_1)}\cdots \xi_{(l_s)}.$$
\end{remark}

\begin{definition}
Let $A$ and $B$ be two dg algebras with auxiliary gradings.
An $\Ainfty$-morphism $\psi: A\to B$ is said to be compatible with the auxiliary gradings if $\psi_r$
 maps $ A_{(k_1)}\otimes \cdots \otimes A_{(k_r)}$ to $B_{(k_1 + \cdots k_r - r +1)}.$
\end{definition}

\begin{remark}
Let $\psi:A \to B$ be  an $\Ainfty$-morphism compatible with auxiliary gradings on $A$ and $B$, respectively. Then $\psi$ restricts to a map:
\[ \psi: F_2(A)\rightarrow F_2(B),\]
which is strongly filtration preserving and therefore there is a well defined push forward map
$$ \psi_*: \MC(F_2(A)) \to \MC(F_2(B)).$$
\end{remark}

\begin{proposition}\label{proposition: HPT}
Let $A$ and $B$ be two dg algebras, equipped with auxiliary gradings. Suppose that
$\psi: A\to B$
is an $\Ainfty$ quasi-isomorphism which is compatible with the auxiliary gradings.
Then the push-forward map
$$ \psi_*: \MC(F_2(A)) \to \MC(F_2(B))$$
is surjective on gauge-equivalence classes of $\MC(F_2(B))$.

\end{proposition}

\begin{proof}
Let $\tau$ be a Maurer-Cartan element of $F_2(B)$.
We will work inductively with respect to the auxiliary grading.
It suffices to prove the following statement for any $k\geq 1$.
Given
$$ \mu[k] := \sum_{2\le l \le k} \mu_{(l)} \in A^1, \quad \textrm{with} \quad \mu_{(l)}\in A^1_{(l)}$$ such that:
\begin{itemize}
 \item[(i)] For all $2\le l \le k$ the equation
 $$ d\mu_{(l)} + \sum_{r+s=l+1} \mu_{(r)}\cdot \mu_{(s)}=0$$
 holds.
 \item[(ii)] 
we have $ \psi_*(\mu[k]) \equiv \tau \,\, \mathrm{mod}(F_{k+1}(B))$.
\end{itemize}
Then there exists $\mu_{(k+1)} \in A^1_{(k+1)}$ and $\xi_{(k)}\in B^0_{(k)}$  such that 
\[\mu[k+1]:=\mu[k]+\mu_{(k+1)},\] 
satisfies the Maurer-Cartan equation modulo $F_{k+3}(A)$ and
\[ \psi_*(\mu[k+1])= \tau \bullet e^{\xi_{(k)}} \,\, \mathrm{mod}  (F_{k+2}(B)).\]
Note that the claim above is indeed enough to prove the proposition, because the infinite product
\[ e^{\xi_{(1)}} e^{\xi_{(2)}} e^{\xi_{(3)}}\cdots =:g\] 
and the infinite sum
\[ \sum_{k\geq 2} \mu_{(k)}=: \mu\]
are convergent and satisfy
$\psi_*(\mu)= \tau \ast g$.

We will now prove the claim above. 
Consider the quantity:
\[ R_{(k+2)}:= \sum_{r+s=k+2}\mu_{(r)}\cdot \mu_{(s)}.\]
The induction hypothesis (i) implies that $R_{(k+2)}$ is closed.
 We claim that the cohomology class of $R_{(k+2)}$ is zero. In order to prove this, define $X_{(k+2)}$ to be the component in $B_{(k+2)}$ of the sum:
$$ \sum_{p,q}\psi_{p+q+1}(( \mu[k])^{\otimes p} \otimes  (d\mu[k] + \mu[k]\cdot \mu[k]) \otimes ( \mu[k])^{\otimes q}) \in B.$$
On the one hand, the induction hypothesis (i) guarantees that
$$ d \mu[k] + \mu[k] \cdot \mu[k] = R_{(k+2)} \quad \mathrm{mod} ( F_{k+3}(A)),$$
which implies that
$$ X_{k+2}= \psi_1(R_{(k+2)}).$$
On the other hand, since $\psi$ is an $\Ainfty$-morphism, $X_{(k+2)}$ equals the $ B_{(k+2)}$-component of
$$ d\psi_*(\mu[k]) + \psi_*(\mu[k])\cdot \psi_*(\mu[k]),$$
which equals the $B_{(k+2)}$-component of
$$ d\psi_*(\mu[k]) + \sum_{r+s=k+2}\tau_{(r)}\cdot \tau_{(s)} = d \left(\psi_*(\mu(k)) - \tau_{(k+1)}\right),$$
We conclude that 
$$ \psi_1(R_{(k+2)}) = X_{(k+2)} = d \left((\psi_*(\mu[k]))_{(k+1)} - \tau_{(k+1)} \right),$$
and in particular that the cohomology class of
$\psi_1(R_{(k+2)})$
is trivial. Since $\psi_1$ induces an isomorphism in cohomology, this implies that the cohomology class
of $R_{(k+2)}$ is trivial as well.
Therefore there exists $\mu_{(k+1)}$ such that
$$ d\mu_{(k+1)} = -R_{(k+2)} = -\sum_{r+s=k+2}\mu_{(r)}\cdot \mu_{(s)},$$
as desired.
Moreover, the component of
$\tau_{(k+1)} - \psi_*(\mu[k]+\mu_{(k+1)})$
in $B_{(k+1)}$ is closed since
\begin{eqnarray*}
 d\left(  (\psi_*(\mu[k]+\mu_{(k+1)}]))_{(k+1)}-\tau_{(k+1)} \right) &=& d(\psi_*(\mu[k])_{(k+1)}) + d (\psi_1(\mu_{(k+1)}) -d(\tau_{(k+1)})\\
 &=&X_{(k+2)} + \psi_1 (d \mu_{(k+1)})=X_{(k+2)}-\psi_1(R_{(k+2)})=0.\end{eqnarray*}

Because $\psi_1$ induces an isomorphism in cohomology, we can find a cocycle $\lambda_{(k+1)} \in A$
and an element $\xi_{(k)} \in B_{(k)}$ 
such that
$$ \tau_{(k+1)} +d \xi_{(k)}=\psi_*(\mu[k]+\mu_{(k+1)})_{(k+1)} +\psi_1(\lambda_{(k+1)}) .$$
holds.
Finally we set
$ \mu'[k+1]:= \mu[k]+ \mu_{(k+1)}+ \lambda_{(k+1)}$
 and observe that by construction we have
 \[ \psi_*(\mu'[k+1])= \tau \bullet e^{\xi_{(k)} }\,\, \mathrm{mod}\, (F_{k+2}(B)).\]
 This completes the inductive argument and hence the proof.

\end{proof}

\end{document}